\documentclass[11pt,a4paper]{article}
\usepackage{amssymb}
\usepackage{amsmath}
\usepackage{latexsym}
\usepackage{graphicx} 
\usepackage{hyperref}
\usepackage{amsthm}
\usepackage{verbatim}
\usepackage{psfrag}
\usepackage{bbm}
\usepackage{dsfont}
\usepackage{mathrsfs}
\usepackage{tikz,pgfplots}
\usepackage{yfonts}
\usepackage{color}
\usepackage{enumerate}
\usepackage{subcaption}
\usepackage{caption}
\usepackage{pgfplots}
\usepackage{subcaption}

\newcommand*{\collauthor}[2]{{#1}$^{#2}$}

\newcommand*{\affiliation}[2]{$\mbox{}^{{#2}}${#1}}
\newcommand*{\colltitle}[1]{\textbf{#1}}

\newtheorem{pro}{Proposition}[section]
\newtheorem{theorem}[pro]{Theorem}

\newenvironment{keywords}[1]{\vspace{1cm}\\{\bf \slshape{Keywords}}\quad\slshape{#1}}{}

\newcommand{\nkappa}{\kappa_\nu}
\newcommand{\nmu}{\mu_\nu}
\newcommand{\nsigma}{\sigma_\nu}

\newcommand{\forceindent}{\leavevmode{\parindent=2em\indent}}

\begin{document}
\begin{center}
\begin{Large}
  \colltitle{A Review of Tree-based Approaches to solve Forward-Backward Stochastic Differential Equations}
\end{Large} 
\vspace*{1.5ex}

\begin{sc}
\begin{large}
\collauthor{Long Teng}{}
\end{large}
\end{sc}
\vspace{1.5ex}

\affiliation{Lehrstuhl f\"ur Angewandte Mathematik und Numerische Analysis,\\
Fakult\"at f\"ur Mathematik und Naturwissenschaften,\\
Bergische Universit\"at Wuppertal, Gau{\ss}str. 20, 
42119 Wuppertal, Germany,\linebreak teng@math.uni-wuppertal.de}{} \\

(This version: Jun 2019)
\end{center}
\section*{Abstract}
In this work, we study solving (decoupled) forward-backward stochastic differential equations (FBSDEs) numerically using the regression trees.
Based on the general theta-discretization for the time-integrands, we show how to efficiently use regression tree-based methods to solve the resulting conditional expectations.
Several numerical experiments including high-dimensional problems are provided to demonstrate the accuracy and performance of the tree-based approach.
For the applicability of FBSDEs in financial problems, we apply our tree-based approach to the Heston stochastic volatility model, the high-dimensional pricing problems of a
Rainbow option and an European financial derivative with different interest rates for borrowing and lending.
\begin{keywords}
forward-backward stochastic differential equations (FBSDEs), high-dimensional problem, regression tree
\end{keywords}

\section{Introduction}
It is well-known that many problems (e.g., pricing, hedging) in the field of financial mathematics can be represented in terms of FBSDEs,
which makes problems easier to solve but exhibits usually no analytical solution, see e.g., \cite{Karoui1997a}.
However, compared to the forward stochastic differential equations (SDEs),
it is more challenged to efficiently find an accurate numerical solution of the FBSDEs. In this work, we show how to solve FBSDEs using the
regression tree-based methods.

The general form of (decoupled) FBSDEs reads
\begin{equation}\label{eq:decoupledbsde}
\left\{
\begin{array}{l}
\,\,\, dX_t = a(t, X_t)\,dt + b(t, X_t)\,dW_t,\quad X_0=x_0,\\
-dY_t = f(t, X_t, Y_t, Z_t)\,dt - Z_t\,dW_t,\\
\quad Y_T=\xi=g(X_T),
 \end{array}\right.
 \end{equation}
where $X_t, a \in \mathbb{R}^n,$ $b$ is a $n\times d$ matrix, $f(t, X_t, Y_t, Z_t): [0, T] \times \mathbb{R}^n \times \mathbb{R}^m \times \mathbb{R}^{m\times d} \to \mathbb{R}^m$
is the driver function and $\xi$ is the square-integrable terminal condition.
We see that the terminal condition $Y_T$ depends on the final value of a forward SDE.
For $a=0~\mbox{and}~b=1,$ namely $X_t=W_t,$ one obtains a backward stochastic differential equation (BSDE) of the form
\begin{equation}\label{eq:bsde}
 \left\{
 \begin{array}{l}
-dY_t = f(t, Y_t, Z_t)\,dt - Z_t\,dW_t,\\
 \quad Y_T = \xi = g(W_T),
 \end{array} \right.
\end{equation}
where $Y_t \in \mathbb{R}^m,~W_t=(W^1_t,\cdots, W^d_t)^T$ is a $d$-dimensional Brownian motion and $f:[0,T]\times\mathbb{R}^m\times\mathbb{R}^{m\times d}\to \mathbb{R}^m.$

The existence and uniqueness of solutions of such equations under the Lipschitz conditions on $f, a(t,X_t), b(t,X_t)~\mbox{and}~g$ are proved
by Pardoux and Peng \cite{Pardoux1990, Pardoux1992}.
Since then, many works try to relax this condition, e.g., the uniqueness of solution is extended under more general assumptions for $f$ in \cite{Lepeltier1997}
but only in the one dimensional case. The solution of a (F)BSDE is a pair of adapted processes $(Y, Z),$ the role of $Z,$ namely $Z_t\,dW_t$ is to render the process
$Y$ be adapted. Moreover, in the application, the process $Z$ can possess some useful information. For example, in option pricing problems, the process $Z$ represents the hedging portfolio while the process $Y$ corresponds to the option price.

In recent years, many numerical methods have been proposed for coupled and decoupled (F)BSDEs. For the numerical algorithms with (least-squares) Monte-Carlo approaches we refer to
\cite{Bender2012, Bouchard2004, Gobet2005, Lemor2006, Zhao2006},
the multilevel Monte Carlo method based on Picard approximation for high-dimensional nonlinear BSDEs can be found in \cite{E2019}.
Some numerical methods for BSDEs applying binomial tree are investigated in \cite{Ma2002}.
There exists connection between BSDEs and PDEs, see \cite{Karoui1997b, Peng1991}, some numerical schemes with the aid of this connection can be found e.g., in \cite{Douglas1996, Ma1994, Milstein2006}.
For the deep-learning-based numerical method we refer to \cite{E2017}.
The approach based on the Fourier method for BSDEs is developed in \cite{Ruijter2015}. See also
\cite{Crisan2010} for the numerical schemes using cubature methods and \cite{Teng2018} for the tree-based approach. And many others e.g., \cite{Bally1997, Bender2008, Fu2017, Gobet2010, Ma2009, Ma2005, Zhang2004, Zhang2013, Zhao2010, Zhao2014}.

In this paper, we show how to efficiently use regression tree-based approaches to find accurate approximations of (F)BSDEs \eqref{eq:decoupledbsde} and \eqref{eq:bsde}.
We apply the general theta-discretization method for the time-integrands and approximate the resulting conditional expectations
using the regression tree-based approach.
The schemes with different theta values are analyzed for the tree-based approach.
Several numerical experiments of different types including high-dimensional problems and applications in pricing financial derivatives are performed to demonstrate our findings.
We show numerical examples of $100$-dimensional FBSDE to check the performance and applicability of our tree-based approach for a high-dimensional problem.

In the next section, we start with notation and definitions and discuss in Section 3
the discretization of time-integrands using the theta-method, and derive the reference equations according to the tree-based method.
Section 4 is devoted to how to use the regression tree-based approaches to approximate the conditional expectations.
In Section 5, several numerical experiments on different types of (F)BSDEs including financial applications are provided to show the accuracy and applicability for
high-dimensional problems.
Finally, Section 6 concludes this work.

\section{Preliminaries}
Throughout the paper, we assume that $(\Omega, \mathcal{F}, P; \{\mathcal{F}_t\}_{0\leq t\leq T})$ is a complete, filtered probability space.
In this space, a standard $d$-dimensional Brownian motion $W_t$ with a finite terminal time $T$ is defined, which generates the filtration $\{\mathcal{F}_t\}_{0\leq t\leq T},$ i.e., 
 $\mathcal{F}_t=\sigma\{X_s, 0\leq s \leq t\}$ for FBSDEs or $\mathcal{F}_t=\sigma\{W_s, 0\leq s \leq t\}$ for BSDEs. And the usual hypotheses should be satisfied. We denote the set of all $\mathcal{F}_t$-adapted and square
integrable processes in $\mathbb{R}^d$ with $L^2=L^2(0,T;\mathbb{R}^d).$
A pair of process $(Y_t,Z_t)$ is the solution of the (F)BSDEs \eqref{eq:decoupledbsde} or \eqref{eq:bsde}
if it is $\mathcal{F}_t$-adapted and square integrable and satisfies \eqref{eq:decoupledbsde} or \eqref{eq:bsde} as
\begin{equation}\label{eq:generalBSDE_01}
 Y_t = \xi + \int_{t}^T f(s, (X_s), Y_s, Z_s)\,ds - \int_t^T Z_s\,dW_s,\quad t \in [0, T],
\end{equation}
where $f(t, (X_s), Y_s, Z_s): [0, T]~(\times \mathbb{R}^n) \times \mathbb{R}^m \times \mathbb{R}^{m\times d} \to \mathbb{R}^m$ is $\mathcal{F}_t$ adapted,
$\xi=g(X_T): \mathbb{R}^n \to \mathbb{R}^m ~\mbox{or}~\xi=g(W_T): \mathbb{R}^d \to \mathbb{R}^m.$ As mentioned above, these solutions exist uniquely under Lipschitz conditions. 

Suppose that the terminal value $Y_T$ is of the form $g(X^{t,x}_T),$ where $X^{t,x}_T$ denotes the solution of $dX_t$ in \eqref{eq:decoupledbsde} starting
from $x$ at time $t.$
Then the solution $(Y^{t,x}_t, Z^{t,x}_t)$ of FBSDEs \eqref{eq:decoupledbsde} can be represented
\cite{Karoui1997b, Ma2005, Pardoux1992, Peng1991} as
\begin{equation}\label{eq:pderelation}
 Y^{t,x}_t = u(t, x), \quad Z^{t,x}_t=(\nabla u(t, x))b(t, x) \quad \forall t \in [0, T),
\end{equation}
which is solution of the semi-linear parabolic PDE of the form
\begin{equation}
\frac{\partial u}{\partial t} +  \sum_i^na_i \partial_i u + \frac{1}{2}\sum_{i,j}^n (bb^T)_{i,j} \partial^2_{i,j} u + f(t, x, u, (\nabla u)b)=0
\end{equation}
with the terminal condition $u(T,x)=g(x).$ Clearly, the corresponding PDE to the BSDEs \eqref{eq:bsde} with $\xi=g(W_T):\Omega \times \mathbb{R}^d \to \mathbb{R}^m$ reads
\begin{equation}
\left\{
\begin{array}{r}
 \frac{\partial u}{\partial t} + \frac{1}{2}\sum_i^d \partial^2_{i,i} u + f(t, u, (\nabla u)b)=0,\\
 \quad\quad u(T,x)=g(x).
\end{array}\right.
 \end{equation}
In turn, suppose $(Y, Z)$ is the solution of (F)BSDEs, $u(t, x)=Y^{t,x}_t$ is a viscosity solution to the PDEs.
As mentioned above, BSDE is a special case of FBSDE with $a=0~\mbox{and}~b=1.$ Thus, we introduce the numerical schemes
concerning FBSDEs in the sequel.
\section{Discretization of the FBSDE using theta-method}
For simplicity, we discuss the discretization with one-dimensional processes, namely $m=n=d=1.$
And the extension to higher dimensions is possible and straightforward.
We introduce the time partition for the time interval $[0, T]$
\begin{equation}\label{eq:timepartition}
\Delta_t = \{t_i | t_i \in [0, T], i=0,1,\cdots,N_T, t_i<t_{i+1}, t_0=0, t_{N_T}=T \}.
 \end{equation}
Let $\Delta t_i = t_{i+1}-t_i$ be the time step, and denote the maximum time step with $\Delta t.$
For the FBSDEs, one needs to additionally discretize the forward SDE in \eqref{eq:decoupledbsde} 
\begin{equation}\label{eq:X_int}
 X_t = x_0 + \int_0^t a(s, X_s)\,ds + \int_0^t b(s, X_s)\,dW_s.
\end{equation}
Suppose that the forward SDE \eqref{eq:X_int} can be already discretized by a process $X^{\Delta_t}_{t_i}$
such that
\begin{equation}
 E\left[\max_{t_i}\left|X_{t_i}-X^{\Delta_t}_{t_i}\right|^2\right]=\mathcal{O}({\Delta_t})
\end{equation}
which means strong mean square convergence of order $1/2.$
In the case of that $X_t$ follows a known distribution (e.g., geometric Brownian motion), one can obtain good
samples on $\Delta_t$ using the known distribution, otherwise the Euler scheme can be employed.

Then one needs to discretize the backward process \eqref{eq:generalBSDE_01}, namely
\begin{equation}\label{eq:generalBSDE_02}
 Y_t = \xi + \int_{t}^T f(s, \mathbb{X}_s)\,ds - \int_t^T Z_s\,dW_s, \quad t\in[0, T),
\end{equation}
where $\xi=g(X_T), \mathbb{X}_s=(X_s, Y_s, Z_s).$ 
Let $(Y_t, Z_t)$ be the adapted solution of \eqref{eq:generalBSDE_02}, we thus have
\begin{equation}\label{eq:generalBSDE_03}
 Y_i = Y_{i+1} + \int_{t_i}^{t_{i+1}} f(s, \mathbb{X}_s)\,ds - \int_{t_i}^{t_{i+1}} Z_s\,dW_s,
\end{equation}
where $Y_{t_i}$ is denoted by $Y_i$ for simple notation.
To obtain adaptability of the solution $(Y_t, Z_t),$ we could use conditional expectations $E_i[\cdot](=E[\cdot|\mathcal{F}_{t_i}]).$
We consider firstly to find the reference equation for $Z.$ By multiplying both sides of the equation \eqref{eq:generalBSDE_03} by $\Delta W_{i+1}:=W_{t_{i+1}} - W_{t_{i}}$
and taking the conditional expectations $E_i[\cdot]$ on the both sides of the derived equation we obtain
\begin{equation}\label{eq:disBSDEZ_01}
 -E_i[Y_{i+1}\Delta W_{i+1}] = \int_{t_{i}}^{t_{i+1}}E_i[f(s,\mathbb{X}_s)\Delta W_s]\,ds - \int_{t_{i}}^{t_{i+1}}E_i[Z_s]\,ds,
\end{equation}
where the It\^o isometry and Fubini's theorem are used and $\Delta W_s:=W_s - W_{t_i}.$
Obviously, with respect to the filtration $\mathcal{F}_{t_i},$ the integrands on the right-hand side of \eqref{eq:disBSDEZ_01}
is deterministic of time $s.$ Thus, applying the theta-method gives
\begin{equation}\label{eq:disBSDEZ_02}
\begin{split}
  &-E_i[Y_{i+1}\Delta W_{i+1}] = \Delta t_i(1-\theta_1)E_i[f(t_{i+1},\mathbb{X}_{i+1})\Delta W_{i+1}]-\Delta t_i \theta_2 Z_i\\
  &\hspace*{2cm}-\Delta t_i(1-\theta_2)E_i[Z_{i+1}] + R_{\theta}^{Z_i},\\
  &\approx\Delta t_i(1-\theta_1)E_i[f(t_{i+1},\mathbb{X}_{i+1})\Delta W_{i+1}]-\Delta t_i \theta_2 Z_i
  -\Delta t_i(1-\theta_2)E_i[Z_{i+1}],
\end{split}
\end{equation}
where $\theta_1 \in[0, 1], \theta_2\in[0,1)$ and $R_{\theta}^{Z_i}$ is the discretization error of the integrals in \eqref{eq:disBSDEZ_01}. Therefore, the equation \eqref{eq:disBSDEZ_02}
lead to a time discrete approximation $Z^{\Delta_t}$ for $Z$
\begin{equation}\label{eq:disBSDEZ_03}
 \begin{split}
  Z_i^{\Delta_t}&=\frac{\theta^{-1}_2}{\Delta t_i}E_i[Y_{i+1}^{\Delta_t}\Delta W_{i+1}] + \theta^{-1}_2(1-\theta_1)E_i[f(t_{i+1},\mathbb{X}^{\Delta_t}_{i+1})\Delta W_{i+1}]\\
  &-\theta^{-1}_2(1-\theta_2)E_i[Z_{i+1}^{\Delta_t}].
 \end{split}
\end{equation}

We start now finding the reference equation for $Y.$
We could take the conditional expectations $E_i[\cdot]$ on the both sides of \eqref{eq:generalBSDE_02} to obtain
\begin{equation}\label{eq:disBSDEY_01}
 Y_i = E_i[Y_{i+1}] + \int_{t_i}^{t_{i+1}} E_i[f(s, \mathbb{X}_s)]\,ds.
\end{equation}
Again, the integrand on the right-hand side of \eqref{eq:disBSDEY_01}
is deterministic of time $s$ with respect to the filtration $\mathcal{F}_{t_i}.$  We use again the theta-method and obtain
\begin{equation}\label{eq:yieEiYi1}
 \begin{split}
  Y_i & = E_i[Y_{i+1}] + \Delta{t_i} \theta_3 f(t_i, \mathbb{X}_i) + \Delta{t_i} (1-\theta_3) E_i[f(t_{i+1}, \mathbb{X}_{i+1})] + R_{\theta}^{Y_i} ,\quad \theta_3\in[0, 1]\\
      & \approx E_i[Y_{i+1}] + \Delta{t_i} \theta_3 f(t_i, \mathbb{X}_i) + \Delta{t_i} (1-\theta_3) E_i[f(t_{i+1}, \mathbb{X}_{i+1})],
 \end{split}
\end{equation}
where $R_{\theta}^{Y_i}$ is the discretization error of the integral in \eqref{eq:generalBSDE_02}. 
Due to $\mathbb{X}^{\Delta_t}_i=(X_i^{\Delta_t}, Y_i^{\Delta_t}, Z_i^{\Delta_t}),$ obviously, we have obtained an implicit scheme which can be directly solved by using
iterative methods, e.g., Newton's method or Picard scheme. 

By choosing the different values for $\theta_1$ and $\theta_2,$ one can obtain different schemes.
For example, one receives the Crank-Nicolson scheme by setting $\theta_1=\theta_2=\theta_3=1/2,$ which is second-order accurate.
When $\theta_1=\theta_2=\theta_3=1,$ the scheme is first-order accurate, see \cite{Zhao2006, Zhao2009, Zhao2013}.
In our experiments we find that the numerical second-order convergence rate can only be achieved when the number of samples
is sufficiently large. The convenience rate of the tree-based method is one divided by the square root of sample size,
to receive the accuracy $(\Delta t)^{2}=(\frac{T}{N_T})^{2},$ the number of samples should be around
$(\frac{N_T}{T})^4.$ For example, when $T=0.5~\mbox{and}~N_T=32,$
one needs $64^4$ samples to obtain that accuracy, that is a quite large integer.
Therefore, to evaluate the performance of the tree-based methods with smaller sample size, in this work we will consider the first-order accurate scheme
for solving the FBSDEs by choosing $\theta_1=1/2, \theta_2=1, \theta_3=1/2$:
 \begin{eqnarray}
  Y_{N_T}^{\Delta_t} &=& g(X_{N_T}^{\Delta_t}),\,Z_{N_T}^{\Delta_t}=g_x(X_{N_T}^{\Delta_t}),\label{eq:scheme_01}\\
  ~\mbox{\bf For}~i&=&N_T-1,\cdots,0: \nonumber\\
  Z_i^{\Delta_t}&=&\frac{1}{\Delta t_i}E_i[Y_{i+1}^{\Delta_t}\Delta W_{i+1}] + \frac{1}{2} E_i[f(t_{i+1},\mathbb{X}^{\Delta_t}_{i+1})\Delta W_{i+1}],\label{eq:scheme_02}\\
 Y_i^{\Delta_t}&=&E_i[Y_{i+1}^{\Delta_t}] + \frac{\Delta{t_i}}{2} f(t_i, \mathbb{X}^{\Delta_t}_i) + 
 \frac{\Delta{t_i}}{2} E_i[f(t_{i+1}, \mathbb{X}^{\Delta_t}_{i+1})].\label{eq:scheme_03}
 \end{eqnarray}
The error estimates for the scheme above is given in Section \ref{sec:praticalapp}.

 \section{Computation of conditional expectations with the tree-based approach}
In this section we introduce how to use the tree-based approach to compute the
conditional expectations included in the schemes introduced above, which actually are all of the form $E[Y|X]$ for square integrable
random variables $X$ and $Y.$ Therefore, we present the regression approach based on the form $E[Y|X]$ throughout this section.
 \subsection{Non-parametric regression}
We assume that the model in non-parametric regression reads
\begin{equation}
 Y = \eta(X) + \epsilon,
\end{equation}
where $\epsilon$ has a zero expectation and a constant variance. Obviously, it can be thus implied that
\begin{equation}\label{eq:conditioneta}
 E[Y|X=x]=\eta(x).
\end{equation}
To approximate the conditional expectations, our goal in regression is to find the estimator of this function, $\hat{\eta}(x).$
By non-parametric regression, we are not assuming a particular form for $\eta.$ Instead of, $\hat{\eta}$ is represented in a regression tree.
Suppose we have a set of samples, $(\hat{x}_\mathcal{M}, \hat{y}_\mathcal{M}),\,\mathcal{M}=1,\cdots,M,$ for $(X, Y),$ where $X$ denotes a predictor variable and $Y$ presents the
corresponding response variable. With such samples we construct a regression tree, which can then be used to determine $E[Y|X=x]$ for
an arbitrary $x,$ whose value is not necessarily equal to one of samples $\hat{x}_\mathcal{M}.$ 

As an example, we specify the procedure for \eqref{eq:scheme_02} in case of FBSDEs, namely where $\mathbb{X}^{\Delta_t}_{i+1}=(X_{i+1}^{\Delta_t}, Y_{i+1}^{\Delta_t}, Z_{i+1}^{\Delta_t}).$
We assume that $(X_{i}^{\Delta_t}, \mathcal{F}_{t_i})_{t_i\in \Delta_t}$ is Markovian. Hence, \eqref{eq:scheme_02} can be rewritten as
\begin{equation}\label{eq:example_z}
   Z_i^{\Delta_t}=E\left[\frac{1}{\Delta t_i} Y_{i+1}^{\Delta_t}\Delta W_{i+1} + \frac{1}{2}f(t_{i+1},\mathbb{X}^{\Delta_t}_{i+1})\Delta W_{i+1}|X_{i}^{\Delta_t}\right],\quad i=N_T-1,\cdots,0.
\end{equation}
And there exist deterministic functions $z_i^{\Delta_t}(x)$ such that
\begin{equation}
Z_i^{\Delta_t}=z_i^{\Delta_t}(X_{i}^{\Delta_t}). 
\end{equation}
Starting from the time $T,$ we construct the regression tree $\hat{T}_z$ for the conditional expectation in 
\eqref{eq:example_z} using samples $(\hat{x}_{N_T-1, \mathcal{M}}, \frac{1}{\Delta t_{N_T-1}}\hat{y}_{N_T, \mathcal{M}}\Delta\hat{w}_{N_T,\mathcal{M}}+\frac{1}{2}\hat{f}_{N_T, \mathcal{M}} \Delta\hat{w}_{N_T,\mathcal{M}}).$ Thereby, the function
\begin{equation}\label{eq:zfun}
 z_{N_T-1}^{\Delta_t}(x) = E\left[\frac{1}{\Delta t_{N_T-1}}Y_{N_T}^{\Delta_t}\Delta W_{N_T} + f(t_{N_T},\mathbb{X}^{\Delta_t}_{N_T})\Delta W_{N_T}|X_{N_T-1}^{\Delta_t}=x\right],
\end{equation}
is estimated and presented by a regression tree. Based on the constructed tree, by applying \eqref{eq:zfun} to the samples $\hat{x}_{N_T-1, \mathcal{M}}$ one can directly obtain
the samples $\hat{z}_{N_T-1, \mathcal{M}}$ of the random variable $Z_{N_T-1}^{\Delta_t},$ for $\mathcal{M}=1,\cdots, M.$
Recursively, backward in time, these samples $\hat{z}_{N_T-1, \mathcal{M}}$ will be used to generate samples $\hat{z}_{N_T-2, \mathcal{M}}$ of the random variables $Z_{N_T-2}^{\Delta_t}$ at the time $t_{N_T-2}.$
At the initial time $t=0,$ we have a fix initial value $x_0$ for $dX_t,$ no samples are needed. Using the regression trees constructed at time $t_1$
we obtain the solution $Z_{0}^{\Delta_t}=z_{0}^{\Delta_t}(x_0).$ For the BSDEs, $X_t$ is just the Brownian motion $W_t,$ which has the zero initial value. 
Following the same procedure to the conditional expectations in \eqref{eq:scheme_03}, one obtains implicitly $Y_{0}^{\Delta_t}.$

\subsection{Binary regression tree}\label{sub:binaryreg}
As mentioned above, regression tree is used to estimate relationship between the predictor variable $X$ and the response variable $Y,$
namely to find the estimator $\hat{\eta}$ of $\eta$ in \eqref{eq:conditioneta} and then to predict given future samples of $X.$ In this section, we review the procedure in \cite{Breiman1984, Martinez2007}
for constructing a best regression tree based on the given samples. Basically, we need to grow, prune and finally select the tree.
We firstly give the notation:
\begin{itemize}
 \item $(\hat{x}_\mathcal{M}, \hat{y}_\mathcal{M})$ denote samples, namely observed data.
\item $\hat{t}$ is a node in the tree $\hat{T},$ $\hat{t}_L~\mbox{and}~\hat{t}_R$ are the left and right child nodes.
\item $\mathcal{T}$ is the set of terminal nodes in the the tree $\hat{T}$ with the number $|\mathcal{T}|$
\item $n(\hat{t})$ represents the number of samples in node $\hat{t}.$
\item $\bar{y}(\hat{t})$ is the average of samples falling into node $\hat{t},$ namely predicted response
 \end{itemize}
\paragraph{Growing a Tree}
We define predicted response as the average value of the samples which are contained in a node $\hat{t},$ namely
\begin{equation}
 \bar{y}(\hat{t}) = \frac{1}{n(\hat{t})}\sum_{\hat{x}_\mathcal{M} \in \hat{t}}\hat{y}_\mathcal{M}.
\end{equation}
Obviously, the squared error in the node $\hat{t}$ reads
\begin{equation}
 R(\hat{t}) = \frac{1}{n(\hat{t})}\sum_{\hat{x}_\mathcal{M} \in \hat{t}}(\hat{y}_\mathcal{M}-\bar{y}(\hat{t}))^2.
\end{equation}
The mean squared error for the tree $\hat{T}$ is defined as the sum of the squared errors in all the terminal nodes and given by
\begin{equation}\label{eq:error}
 R(\hat{T})=\sum_{\hat{t}\in\mathcal{T}} R(\hat{t})= \frac{1}{n(\hat{t})}\sum_{\hat{t}\in\mathcal{T}}\sum_{\hat{x}_\mathcal{M} \in \hat{t}}(\hat{y}_\mathcal{M}-\bar{y}(\hat{t}))^2.
\end{equation}
Basically, the tree is constructed by partitioning the space for the samples $\hat{x}$ using a sequence of binary splits. For a split $s~\mbox{and node}~\hat{t},$
the change in the mean squared error can be thus calculated as
\begin{equation}
 \Delta R(s, \hat{t}) = R(\hat{t}) - R(\hat{t}_L) - R(\hat{t}_R).
\end{equation}
The regression tree is thus obtained by
iteratively splitting nodes with $s,$ which yields the largest $\Delta R(s, \hat{t}).$ Thereby, decrease in $R(\hat{T})$ is maximized.
Obviously, the optimal stopping criterion is that all responses in a terminal node are the same, but that is not really realistic. There are some other
criteria are available, e.g., growing the tree until number of samples in a terminal node is five, which is suggested in \cite{Breiman1984}.

\paragraph{Pruning a tree}
When using the optimal stopping criterion, all responses in a terminal node are same, i.e., each terminal node contains only one response, then the error $R(\hat{t}),$ therewith $R(\hat{T}),$ will be zero.
However, first of all, this is unrealistic as already mentioned. Secondly, the samples is thereby over fitted and the regression tree will thus not generalize well to new observed samples. 
Breiman et al. \cite{Breiman1984} suggested growing an overly large regression tree $\hat{T}_{\mbox{max}}$ and then to find nested sequence of sub-trees by successively pruning
branches of the tree. This procedure is called pruning a tree. We define an error-complexity measure as
\begin{equation}
 R_{\alpha}(\hat{T}) =  R(\hat{T}) + \alpha|\mathcal{T}|, \quad \alpha \geq 0,
\end{equation}
where $\alpha$ represents the complexity cost per terminal node. 
The error-complexity should be minimized by looking for trees.
Let $\hat{T}_{\mbox{max}}$ be the overly large tree, in which each terminal node
contains only one response. Thus, we have $R_{\alpha}(\hat{T}_{\mbox{max}})=\alpha|\mathcal{T}|$ which indicates a high cost of complexity, while the error is small. 
To minimize the cost we delete the branches with the weakest link $\hat{t}_k^*$ in tree $\hat{T}_k$, which is defined
as
\begin{equation}
 g_k(\hat{t}_k^*)=\mbox{min}_{\hat{t}}\{g_k(\hat{t})\},\quad g_k(\hat{t})=\frac{R(\hat{t})-R(\hat{T}_{k\hat{t}})}{|\mathcal{T}_{k\hat{t}}|-1},
\end{equation}
where $\hat{T}_{k\hat{t}}$ is the branch $\hat{T}_{\hat{t}}$ corresponding to the internal node $\hat{t}$ of sub-tree $\hat{T}_k.$
Then, we prune the branch defined by the node $\hat{t}_k^*$
\begin{equation}
 \hat{T}_{k+1}=\hat{T}_k-\hat{T}_{\hat{t}_k^*},
\end{equation}
and thus obtain a finite sequence of sub-trees with fewer terminal nodes and decreasing complexity until the root node as
\begin{equation}\label{eq:seq_1}
 \hat{T}_{\mbox{max}} > \hat{T}_{1} > \hat{T}_{2} > \cdots >\hat{T}_{K}=\mbox{root}.
\end{equation}
On the other hand, we set
\begin{equation}
 \alpha_{k+1}=g_k(\hat{t}_k^*)
\end{equation}
and thus obtain an increasing sequence of values for the complexity parameter $\alpha,$ namely
\begin{equation}\label{eq:seq_2}
 0=\alpha_1<\cdots<\alpha_k<\alpha_{k+1}<\cdots\alpha_K.
\end{equation}
By observing the both sequences \eqref{eq:seq_1} and \eqref{eq:seq_2}, it is not difficult to find:
for $k\geq 1,$ the tree $\hat{T}_k$ is the one which has the minimal cost complexity for $\alpha_k\leq \alpha<\alpha_{k+1}.$

\paragraph{Selecting a Tree}
We have to make a trade-off between the both criteria of error and complexity, namely we need to choose the best tree from the sequence of pruned sub-trees such
that the complexity of tree and squared error are both minimized. To do this, there are two possible ways introduced in \cite{Breiman1984, Martinez2007},
namely independent test samples and cross-validation. As an example, we illustrate the independent test sample method, for cross-validation we refer to \cite{Breiman1984, Martinez2007}. 
Clearly, we need honest estimates of the true error $R^*(\hat{T})$ to select the right size of the tree. To obtain that estimates, we should use samples that were not
used to construct the tree to estimate the error. Suppose we have a set of samples $L=(\hat{x}_\mathcal{M}, \hat{y}_\mathcal{M}),$ which should be randomly divided into two subsets $L_1~\mbox{and}~L_2.$
We use the set $L_1$ to grow a large tree and to obtain the sequence of pruned sub-trees. Thus, the samples in $L_2$ is used to evaluate the performance of each sub-tree by
calculating the error between real response and predicated response. We denote the predicated response using samples $\hat{x}$ to the tree $\hat{T}_k$ with $\bar{y}_k(\hat{x}),$ then the estimated error is
\begin{equation}
 \hat{R}(\hat{T}_k)=\frac{1}{n_2}\sum_{(\hat{x}_i, \hat{y}_i)\in L_2} (\hat{y}_i-\bar{y}_k(\hat{x}_i))^2,
\end{equation}
where $n_2$ is the number of samples in $L_2.$ This estimated error will be calculated for all sub-trees. As mentioned above,
if one directly select the tree with the smallest error, then the cost of complexity will be higher. Instead of, we can pick a sub-tree
that has the fewest number of nodes, but still keeps the accuracy of the tree with the smallest error, say $\hat{T}_0$ with the error $\hat{R}_{\mbox{min}}(\hat{T}_0).$
To do this, we define the standard error for this estimate as \cite{Breiman1984}
\begin{equation}\label{eq:SEhatR}
 SE(\hat{R}_{\mbox{min}}(\hat{T}_0)):=\frac{1}{\sqrt{n_2}}\sqrt{\frac{1}{n_2}\sum_{i=1}^{n_2} (\hat{y}_i-\bar{y}(\hat{x_i}))^4 - (\hat{R}_{\mbox{min}}(\hat{T}_0))^2},
\end{equation}
and then choose the smallest tree $\hat{T}_k^*$ such that
\begin{equation}
 \hat{R}(\hat{T}_k^*)\leq \hat{R}_{\mbox{min}}(\hat{T}_0) + SE(\hat{R}_{\mbox{min}}(\hat{T}_0)).
\end{equation}
$\hat{T}_k^*$ is the tree with minimal complexity cost but has equivalent accuracy as the tree with minimum error.

\subsection{Practical Applications}\label{sec:praticalapp}
Note that we do not need to construct the individual tree for each conditional expectation in the schemes.
Due to the linearity of conditional expectation, we construct the trees for all possible combinations of the conditional expectations.
We denote the tree's regression error with $R_{\mbox{tr}},$ the error of used iterative method with $R_{\mbox{impl}}$ and reformulate the scheme \eqref{eq:scheme_01}-\eqref{eq:scheme_03} by
combining conditional expectations and including all errors as
\begin{align*}
  \hat{y}_{N_T, \mathcal{M}} &= g(\hat{x}_{N_T, \mathcal{M}}),\,\hat{z}_{N_T, \mathcal{M}}=g_x(\hat{x}_{N_T, \mathcal{M}}),\\
  ~\mbox{\bf For}~i&=N_T-1,\cdots,0~,~\mathcal{M}=1,\cdots, M:\\
  \hat{z}_{i, \mathcal{M}}&=E^{\hat{x}_{i, \mathcal{M}}}_{i}\left[\frac{1}{\Delta t_i}Y_{i+1}\Delta W_{i+1} + \frac{1}{2}f(t_{i+1},\mathbb{X}_{i+1})\Delta W_{i+1}\right] + \frac{R^{Z_i}_{\theta}}{\Delta t_i} + R_{\mbox{tr}}^{Z_i},\\
   \hat{y}_{i, \mathcal{M}}&=E^{\hat{x}_{i, \mathcal{M}}}_{i}\left[Y_{i+1} + \frac{\Delta{t_i}}{2} f(t_{i+1}, \mathbb{X}_{i+1})\right] + 
   \frac{\Delta{t_i}}{2} \hat{f}_{i, \mathcal{M}} + R^{Y_i}_{\theta}+R^{Y_i}_{\mbox{impl}}+ R_{\mbox{tr}}^{Y_i},
   \end{align*}
 where $E^{\hat{x}_{i, \mathcal{M}}}_{i}[\mathcal{Y}]$ denotes calculated conditional expectation $E[\mathcal{Y}|X=\hat{x}_{i, \mathcal{M}}]$ using the constructed regression tree with the samples of $\mathcal{Y}.$
 For example, using samples of the predictor variable $X_i$ (which are $\hat{x}_{i, \mathcal{M}}$) and samples of the response variable $\frac{1}{\Delta t_i}Y_{i+1}\Delta W_{i+1} + \frac{1}{2}f(t_{i+1},\mathbb{X}_{i+1})\Delta W_{i+1}$
(which are $\frac{1}{\Delta t_i}\hat{y}_{i+1, \mathcal{M}}\Delta\hat{w}_{i+1,\mathcal{M}}+\frac{1}{2}\hat{f}_{i+1, \mathcal{M}} \Delta\hat{w}_{i+1,\mathcal{M}}$) we construct a regression tree.
 Then, $E^{\hat{x}_{i, \mathcal{M}}}_{i}[\frac{1}{\Delta t_i}Y_{i+1}\Delta W_{i+1} + \frac{1}{2}f(t_{i+1},\mathbb{X}_{i+1})\Delta W_{i+1}]$ means the determined value
 of $E[\frac{1}{\Delta t_i}Y_{i+1}\Delta W_{i+1} + \frac{1}{2}f(t_{i+1},\mathbb{X}_{i+1})\Delta W_{i+1}|X=\hat{x}_{i, \mathcal{M}}]$
 using the constructed tree. Note that, at the initial time $t=0,$ we have $\hat{x}_{0, \mathcal{M}}=x_0~\mbox{for}~\mathcal{M}=1,\cdots, M.$
 
\forceindent
From the errors \eqref{eq:error} and \eqref{eq:SEhatR} we can assume that the approximation error of the tree-based approach is approximately $1/\sqrt{n_2}$
for a large number $n_2=\frac{M}{2},$ which is the number of samples in $L_2$ as introduced above. Theoretically, the regression error can be neglected by choosing
sufficiently high $n_2$, namely $M.$ However, the tree-based approach is computationally not that efficient for a quite high value $M.$
For this our idea is to split a quite large set of samples into several small sets of samples, e.g., we can split a set of $20000$ samples into $10$ sets of $2000$ samples.
The major reason is that the many times tree-based computations for a small sample number are still more efficient than one computation for a large sample number. 
We observe, from $t_{N_T} \to t_1$ in the proposed scheme, the samples of $Y^{\Delta_t}_1~\mbox{and}~Z^{\Delta_t}_1$ are generated backward iteratively starting from the samples
of $Y^{\Delta_t}_{N_T}~\mbox{and}~Z^{\Delta_t}_{N_T}.$ When splitting the samples, this procedure can be seen as the projection of samples from $t_{N_T} \to t_1$ but in different groups.
Moreover, for the step $t_1 \to t_0,$ one has a constant as the predictor variable, namely $W_0=0$ for the BSDE or $X_0=x_0$ for the FBSDE. In fact, in the case of constant predictor, the computation can be done rapidly.
We know that the quality of approximations for $Y^{\Delta_t}_0~\mbox{and}~Z^{\Delta_t}_0$ relies directly on the samples of $Y^{\Delta_t}_1~\mbox{and}~Z^{\Delta_t}_1.$ Our numerical results show that the splitting error of samples projection from $t_{N_T} \to t_1$ could be neglected.

Consequently, we propose to split a large sample size into a few groups of small-size samples at $t_{N_T},$
for each group we generate backward iteratively the samples for $Y^{\Delta_t}_1~\mbox{and}~Z^{\Delta_t}_1$\footnote{ Theoretically, the projection of samples in the different groups can be done parallelly.
However, the parallelization is not considered in this work.}.
Then, at $t_1$ we combine the samples of $Y^{\Delta_t}_1~\mbox{and}~Z^{\Delta_t}_1$ from all groups, which are used as the samples of response variables for the last step $t_1 \to t_0,$ whereas the predictor variable is a constant as mentioned already. Note that in the analysis above we have considered a linear regression model, i.e., the proposed scheme is designed to the linear (F)BSDEs.

We summarize our algorithm to solve the FBSDEs as follows.
\begin{itemize}
 \item Generate $M$ samples and split them into $M_g$ different groups, the sample number in each group is $G=M/M_g.$
  \item For each group, namely $\mathcal{M}=1,\cdots, M_g,$ compute
\begin{align*}
  \hat{y}_{N_T, \mathcal{M}} &= g(\hat{x}_{N_T, \mathcal{M}}),\,\hat{z}_{N_T, \mathcal{M}}=g_x(\hat{x}_{N_T, \mathcal{M}}),\\
  ~\mbox{\bf For}~i&=N_T-1,\cdots,1~,~\mathcal{M}=1,\cdots, M_g:\\
  \hat{z}_{i, \mathcal{M}}&=E^{\hat{x}_{i, \mathcal{M}}}_{i}\left[\frac{1}{\Delta t_i}Y_{i+1}^{\Delta_t}\Delta W_{i+1} + \frac{1}{2}f(t_{i+1},\mathbb{X}^{\Delta_t}_{i+1})\Delta W_{i+1}\right] ,\\
   \hat{y}_{i, \mathcal{M}}&=E^{\hat{x}_{i, \mathcal{M}}}_{i}\left[Y_{i+1}^{\Delta_t} + \frac{\Delta{t_i}}{2} f(t_{i+1}, \mathbb{X}^{\Delta_t}_{i+1})\right] + 
   \frac{\Delta{t_i}}{2} \hat{f}_{i, \mathcal{M}}.
   \end{align*}
 \item Collect all the samples of  $(\hat{z}_{1, \mathcal{M}}, \hat{y}_{1, \mathcal{M}})$ for $\mathcal{M}=1,\cdots, M$ and use all these samples to compute
 \begin{align*}
  Z_{0}^{\Delta_t}&=E^{x_0}_{0}\left[\frac{1}{\Delta t_0}Y_{1}^{\Delta_t}\Delta W_{1} + \frac{1}{2}f(t_{1},\mathbb{X}^{\Delta_t}_{1})\Delta W_{1}\right] ,\\
  Y_{0}^{\Delta_t}&=E^{x_0}_{0}\left[Y_{1}^{\Delta_t} + \frac{\Delta{t_0}}{2} f(t_{1}, \mathbb{X}^{\Delta_t}_{1})\right] + 
   \frac{\Delta{t_0}}{2} \hat{f}_{0, \mathcal{M}}.
   \end{align*}
\end{itemize}

\subsection{Error estimates}
Suppose that $R_{\mbox{tr}}$ and $R_{\mbox{impl}}$ can be neglected by choosing $M$ and Picard iterations sufficiently high, we consider the discretization errors in the first place.
We denote the global errors by
\begin{align}
 \epsilon^{Y_i}(X_{i}^{\Delta_t}):&=Y_i(X_{i}^{\Delta_t})-Y_i^{\Delta_t}(X_{i}^{\Delta_t}),\\
 \epsilon^{Z_i}(X_{i}^{\Delta_t}):&=Z_i(X_{i}^{\Delta_t})-Z_i^{\Delta_t}(X_{i}^{\Delta_t}),\\
 \epsilon^{f_i}(X_{i}^{\Delta_t}):&=f(t_i, \mathbb{X}_{i})-f(t_i, \mathbb{X}_{i}^{\Delta_t}).
\end{align}
Firstly, we give some remarks concerning related results on the one-step scheme:
\begin{itemize}
\item The absolute values of the local errors $R_\theta^{Y_i}~\mbox{and}~R_\theta^{Z_i}$ in \eqref{eq:disBSDEZ_02} and \eqref{eq:yieEiYi1} can be bounded by $C(\Delta t_i)^3$
when $\theta_i=1/2,\,i=1, 2, 3$ and by $C(\Delta t_i)^2$ when $\theta_1=1/2, \theta_2=1, \theta_3=1/2,$ where 
$C$ is a constant which can depend on $T, a, b~\mbox{and functions}~f, g$ in \eqref{eq:decoupledbsde}, see e.g., \cite{Zhao2009, Zhao2012, Zhao2013}.
 \item For notation convenience we might omit the dependency of local and global errors on state of the FBSDEs and the discretization errors of $dX_t,$ namely we assume that
$X_i=X_{i}^{\Delta_t}.$
\item For the implicit schemes we will apply Picard iterations which converges for any initial guess when $\Delta t_i$ is small enough. In the following analysis, we consider the equidistant time discretization $\Delta t.$
\end{itemize}

We start to perform the error analysis for the scheme with  $\theta_1=1/2, \theta_2=1, \theta_3=1/2.$ The error analysis for other choices of $\theta_i$ can be done analogously.
For the $Z$-component $(0\leq i \leq N_T-1)$ we have
\begin{equation}\label{eq:epsilonzi}
\epsilon^{Z_i}=E_i^{x_i}[\frac{1}{\Delta t}\epsilon^{Y_{i+1}}\Delta W_{i+1} + \frac{1}{2}\epsilon^{f_{i+1}} \Delta W_{i+1}] + \frac{R_{\theta}^{Z_i}}{\Delta t},
\end{equation}
where the $\epsilon^{f_{i+1}}$ can be bounded using Lipschitz continuity of $f$ by
\begin{equation}
 E_i^{x_{i}}[|\epsilon^{f_{i+1}}|^2] \leq E_i^{x_{i}}[|L(|\epsilon^{Y_{i+1}}| + |\epsilon^{Z_{i+1}}|)|^2] \leq 2 L^2 E_i^{x_{i}}[|\epsilon^{Y_{i+1}}|^2 + |\epsilon^{Z_{i+1}}|^2]
\end{equation}
with Lipschitz constant $L.$ And it holds that
\begin{equation}
 |E_i^{x_{i}}[\epsilon^{Y_{i+1}} \Delta W_{i+1}]|^2=|E_i^{x_{i}}[(\epsilon^{Y_{i+1}}-E_i^{x_{i}}[\epsilon^{Y_{i+1}}])\Delta W_{i+1}]|^2 \leq \Delta t (E_i^{x_{i}}[|\epsilon^{Y_{i+1}}|^2] -|E_i^{x_{i}}[\epsilon^{Z_{i+1}}]|^2).
\end{equation}
Consequently, we calculate
\begin{equation}\label{eq:error_z}
 (\Delta t)^2|\epsilon^{Z_i}|^2 \leq 6 \Delta t (E_i^{x_{i}}[|\epsilon^{Y_{i+1}}|^2]-|E_i^{x_{i}}[\epsilon^{Y_{i+1}}]|^2)
 +3 L (\Delta t)^3 E_i^{x_{i}}[|\epsilon^{Y_{i+1}}|^2+|\epsilon^{Z_{i+1}}|^2] + 6|R_{\theta}^{Z_i}|^2, 
 \end{equation}
 where H\"older's inequality is used.
 
For the $Y$-component in the implicit scheme we have
\begin{equation}
\epsilon^{Y_i}=E^{x_i}_{i}[\epsilon^{Y_{i+1}} + \frac{\Delta{t}}{2}  \epsilon^{f_{i+1}}] + \frac{\Delta{t}}{2}  \epsilon^{f_{i}} + R^{Y_i}_{\theta},
 \end{equation}
Again using Lipschitz continuity, this error can be bounded by
\begin{equation}\label{eq:error_y}
|\epsilon^{Y_i}|\leq |E^{x_i}_{i}[\epsilon^{Y_{i+1}}]|+ \frac{\Delta t L}{2}(|\epsilon^{Y_{i}}|+|\epsilon^{Z_{i}}|)
 +\frac{\Delta t L}{2}E^{x_i}_{i}[|\epsilon^{Y_{i+1}}|+|\epsilon^{Z_{i+1}}|]+R_{\theta}^{Y_i}.
  \end{equation}
By the inequality $(a+b)^2\leq a^2 + b^2 + \gamma \Delta t a^2 + \frac{1}{\gamma \Delta t} b^2$ we calculate
\begin{equation}\label{eq:error_yy}
\begin{split}
 |\epsilon^{Y_i}|^2 &\leq (1+\gamma\Delta t)|E^{x_i}_{i}[\epsilon^{Y_{i+1}}]|^2 + \frac{3(\Delta t L)^2}{2}(|\epsilon^{Y_{i}}|^2+|\epsilon^{Z_{i}}|^2)+\frac{3(\Delta t L)^2}{2}(|\epsilon^{Y_{i+1}}|^2+|\epsilon^{Z_{i+1}}|^2)\\
 &+3|R_{\theta}^{Y_i}|^2+\frac{1}{\gamma}\left(\frac{3\Delta t L^2}{2}(|\epsilon^{Y_{i}}|^2+|\epsilon^{Z_{i}}|^2)
 +\frac{3\Delta t L^2}{2}(|\epsilon^{Y_{i+1}}|^2+|\epsilon^{Z_{i+1}}|^2)+\frac{3|R_{\theta}^{Y_i}|^2}{\Delta t}\right).
\end{split}
\end{equation}
\begin{theorem}\label{theo:convergence}
Given 
\begin{equation*}
  E_{N_T-1}^{x_{N_T-1}}[|\epsilon^{Z_{N_T}}|^2] \thicksim \mathcal{O}((\Delta t)^2), \quad E_{N_T-1}^{x_{N_T-1}}[|\epsilon^{Y_{N_T}}|^2] \thicksim \mathcal{O}((\Delta t)^{2}),
 \end{equation*}
It holds then
\begin{equation}\label{eq:resultstheorem}
E_0^{x_0}\left[|\epsilon^{Y_{i}}|^2 + \frac{\Delta t}{6}|\epsilon^{Z_{i}}|^2\right]\leq Q (\Delta t)^{2},\quad 0 \leq i \leq N_T-1,
\end{equation}
where $Q$ is a constant which only depend on $T, f, g~\mbox{and}~a, b$ in \eqref{eq:decoupledbsde}.
\end{theorem}
\begin{proof}
 By combining both \eqref{eq:error_z} and \eqref{eq:error_yy} we straightforwardly obtain
 \begin{align}
E_i^{x_i}[|\epsilon^{Y_i}|^2] &+ \frac{\Delta t}{6}E_i^{x_i}[|\epsilon^{Z_i}|^2]
\leq (1+\gamma\Delta t)|E^{x_i}_{i}[\epsilon^{Y_{i+1}}]|^2 + \frac{3(\Delta t L)^2}{2}(E_i^{x_i}|\epsilon^{Y_{i}}|^2+E_i^{x_i}|\epsilon^{Z_{i}}|^2)\nonumber\\
&+\frac{3(\Delta t L)^2}{2}(E_i^{x_i}|\epsilon^{Y_{i+1}}|^2+E_i^{x_i}|\epsilon^{Z_{i+1}}|^2)+3E_i^{x_i}[|R_{\theta}^{Y_i}|^2]\nonumber\\
 &+\frac{1}{\gamma}\left(\frac{3\Delta t L^2}{2}(E_i^{x_i}|\epsilon^{Y_{i}}|^2+E_i^{x_i}|\epsilon^{Z_{i}}|^2)
 +\frac{3\Delta t L^2}{2}(E_i^{x_i}|\epsilon^{Y_{i+1}}|^2+E_i^{x_i}|\epsilon^{Z_{i+1}}|^2)+\frac{3E_i^{x_i}[|R_{\theta}^{Y_i}|^2]}{\Delta t}\right)\nonumber\\
 &+ (E_i^{x_{i}}[|\epsilon^{Y_{i+1}}|^2]-|E_i^{x_{i}}[\epsilon^{Y_{i+1}}]|^2)
 +\frac{L}{2} (\Delta t)^2 E_i^{x_{i}}[|\epsilon^{Y_{i+1}}|^2+|\epsilon^{Z_{i+1}}|^2] + \frac{E_i^{x_i}[|R_{\theta}^{Z_i}|^2]}{\Delta t} \nonumber
 \end{align}
 which implies
 \begin{align}
  &\left(1-\frac{3(\Delta t L)^2}{2}-\frac{3\Delta t L^2}{2\gamma}\right)E_i^{x_i}[|\epsilon^{Y_i}|^2] + 
  \left(\frac{\Delta t}{6}-\frac{3(\Delta t L)^2}{2}-\frac{3\Delta t L^2}{2\gamma}\right)E_i^{x_i}[|\epsilon^{Z_i}|^2]\nonumber\\
 & \leq \left(1 + \gamma \Delta t +2(\Delta t L)^2 + \frac{3\Delta t L^2}{2\gamma}\right)E_i^{x_i}[|\epsilon^{Y_{i+1}}|^2] + \left(2(\Delta t L)^2 + \frac{3\Delta t L^2}{2\gamma}\right)E_i^{x_i}[|\epsilon^{Z_{i+1}}|^2] \nonumber\\
 & + 3E_i^{x_i}[|R_{\theta}^{Y_i}|^2] + \frac{3E_i^{x_i}[|R_{\theta}^{Y_i}|^2]}{\gamma \Delta t} + \frac{E_i^{x_i}[|R_{\theta}^{Z_i}|^2]}{\Delta t}.\nonumber
 \end{align}
 We choose $\gamma$ such that
 $\frac{\Delta t}{6} -\frac{3\Delta t L^2}{2\gamma} \geq \frac{3\Delta t L^2}{2\gamma} (i.e. \gamma \geq 18 L^2),$ by which the
 latter inequality can be rewritten as
 \begin{align}
  &\left(1-\frac{3(\Delta t L)^2}{2}-\frac{3\Delta t L^2}{2\gamma}\right)E_i^{x_i}[|\epsilon^{Y_i}|^2] + 
  \left(\frac{\Delta t}{6}-\frac{3(\Delta t L)^2}{2}-\frac{3\Delta t L^2}{2\gamma}\right)E_i^{x_i}[|\epsilon^{Z_i}|^2]\nonumber\\
 & \leq \left(1 + \gamma \Delta t +2(\Delta t L)^2 + \frac{3\Delta t L^2}{2\gamma}\right)E_i^{x_i}[|\epsilon^{Y_{i+1}}|^2] + \left(2(\Delta t L)^2 + \frac{\Delta t}{6} -\frac{3\Delta t L^2}{2\gamma} \right)E_i^{x_i}[|\epsilon^{Z_{i+1}}|^2] \nonumber\\
 & + 3E_i^{x_i}[|R_{\theta}^{Y_i}|^2] + \frac{3E_i^{x_i}[|R_{\theta}^{Y_i}|^2]}{\gamma \Delta t} + \frac{E_i^{x_i}[|R_{\theta}^{Z_i}|^2]}{\Delta t},\nonumber
 \end{align}
 which implies
 \begin{align}
  E_i^{x_i}[|\epsilon^{Y_i}|^2] + \frac{\Delta t}{6}E_i^{x_i}[|\epsilon^{Z_i}|^2] &\leq \frac{1+C \Delta t}{1-C \Delta t}
  \left(E_i^{x_i}[|\epsilon^{Y_{i+1}}|^2] + \frac{\Delta t}{6}E_i^{x_i}[|\epsilon^{Z_{i+1}}|^2]\right)\nonumber\\
  &+3E_i^{x_i}[|R_{\theta}^{Y_i}|^2] + \frac{E_i^{x_i}[|R_{\theta}^{Y_i}|^2]}{6 L^2 \Delta t} + \frac{E_i^{x_i}[|R_{\theta}^{Z_i}|^2]}{\Delta t}.\nonumber
 \end{align}
By induction, we obtain then
\begin{align}
  E_i^{x_i}[|\epsilon^{Y_i}|^2] + \frac{\Delta t}{6}E_i^{x_i}[|\epsilon^{Z_i}|^2] &\leq \left(\frac{1+C \Delta t}{1-C \Delta t}\right)^{N_T-i}
  \left(E_{N_T-1}^{x_{N_T-1}}[|\epsilon^{Y_{N_T}}|^2] + \frac{\Delta t}{6}E_{N_T-1}^{x_{N_T-1}}[|\epsilon^{Z_{N_T}}|^2]\right)\nonumber\\
  &+\sum_{j=i}^{N_T-1}\left(\frac{1+C \Delta t}{1-C \Delta t}\right)^{j-i}\left(3E_i^{x_i}[|R_{\theta}^{Y_j}|^2] + \frac{E_i^{x_i}[|R_{\theta}^{Y_j}|^2]}{6 L^2 \Delta t} + \frac{E_i^{x_i}[|R_{\theta}^{Z_j}|^2]}{\Delta t}\right)\nonumber\\
  &\leq \exp(2CT)
  \left(E_{N_T-1}^{x_{N_T-1}}[|\epsilon^{Y_{N_T}}|^2] + \frac{\Delta t}{6}E_{N_T-1}^{x_{N_T-1}}[|\epsilon^{Z_{N_T}}|^2]\right)\nonumber\\
  &+\exp(2CT)\sum_{j=i}^{N_T-1}\left(3E_i^{x_i}[|R_{\theta}^{Y_j}|^2] + \frac{E_i^{x_i}[|R_{\theta}^{Y_j}|^2]}{6 L^2 \Delta t} + \frac{E_i^{x_i}[|R_{\theta}^{Z_j}|^2]}{\Delta t}\right).\nonumber
 \end{align}
 With the known conditions and bounds of the local errors we complete the proof.
\end{proof}

\section{Numerical experiments}\label{sec:numexp}
In this section we use some numerical examples to show the accuracy of our methods for solving the (F)BSDEs.
As already introduced above, $N_T~\mbox{and}~M$ are the total discrete time steps and sampling number, respectively.
For all the examples, we consider an equidistant time and perform $20$ Picard iterations.
We ran the algorithms $10$ times independently and take average value of absolute error, whereas the two different
seeds are used for every five simulations.
Numerical experiments were performed with an Intel(R) Core(TM) i5-8500 CPU @ 3.00GHz and 15 GB RAM.

\subsection{Example of BSDE}\label{sec:bsdeexample}
The first BSDE we consider is 
\begin{equation}\label{eq:example02}
\left\{
\begin{array}{l}
-dY_t = (\frac{Y_t}{2}-\frac{Z_t}{2})\,dt - Z_t\,dW_t,\\  
 \quad Y_T = \sin(W_T + \frac{T}{2}),
 \end{array}\right.
 \end{equation}
 with the analytical solution
 \begin{equation}\label{eq:example02s}
 \left\{
 \begin{array}{l}
 Y_t = \sin(W_t + \frac{t}{2}),\\  
 Z_t = \cos(W_t + \frac{t}{2}).
 \end{array}\right.
 \end{equation}
The generator $f$ is highly oscillatory function and contains the component $Z_t.$ 
For this example we set $T=\frac{1}{2},$ the analytical solution
of $(Y_0, Z_0)$ is $(0, 1).$

Firstly, in order to see the computational acceleration by using the samples-splitting introduced above, we compare the scheme between
using and not using the samples-splitting in Figure \ref{fig:01}. Since the algorithm without splitting are slow, we thus compare them up
to the sample size 50000, whereas $N_T$ is fixed to $10.$
Let $Y^{\Delta_t}_{0,k}$ and $Z^{\Delta_t}_{0,k}$ denote the result on the $k$-th run of the algorithm, $k=1,\cdots, 10,$
the approximations read as $Y^{\Delta_t}_0=\frac{1}{10}\sum_{k=1}^{10}Y^{\Delta_t}_{0,k}$ and $Z^{\Delta_t}_0=\frac{1}{10}\sum_{k=1}^{10}Z^{\Delta_t}_{0,k}.$
In our tests we consider average of the absolute errors, i.e., $\frac{1}{10}\sum_{k=1}^{10} |Y^{\Delta_t}_{0,k}-Y_0|$ and $\frac{1}{10}\sum_{k=1}^{10} |Z^{\Delta_t}_{0,k}-Z_0|.$
\begin{figure}[htbp!]
 \centering
 \begin{subfigure}[b]{0.47\textwidth}
 \includegraphics[width=\textwidth]{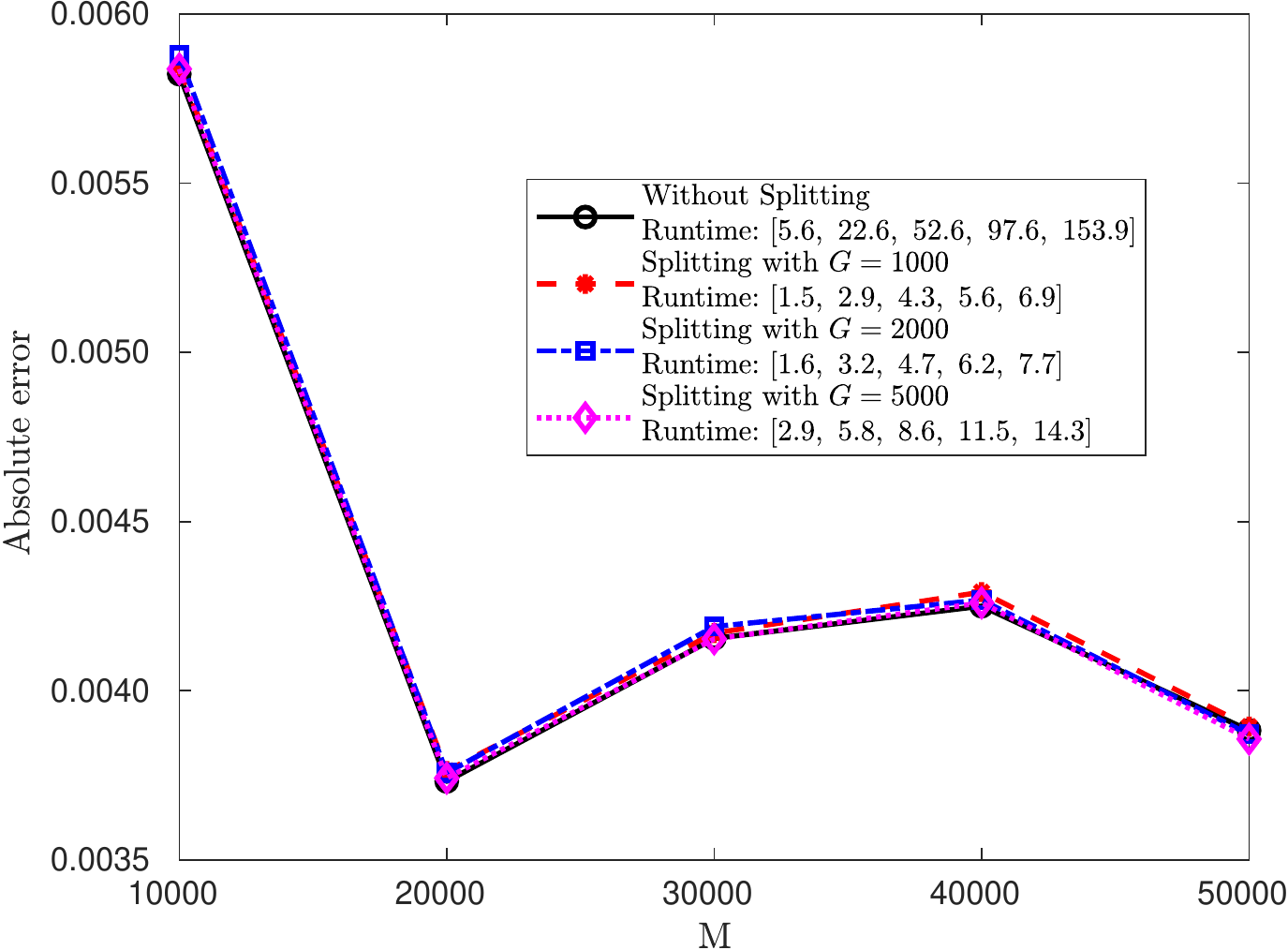}
 \subcaption{Absolute error: $\frac{1}{10}\sum_{k=1}^{10} |Y^{\Delta_t}_{0,k}-Y_0|$}
\end{subfigure}
 ~ 
 \begin{subfigure}[b]{0.47\textwidth}
 \includegraphics[width=\textwidth]{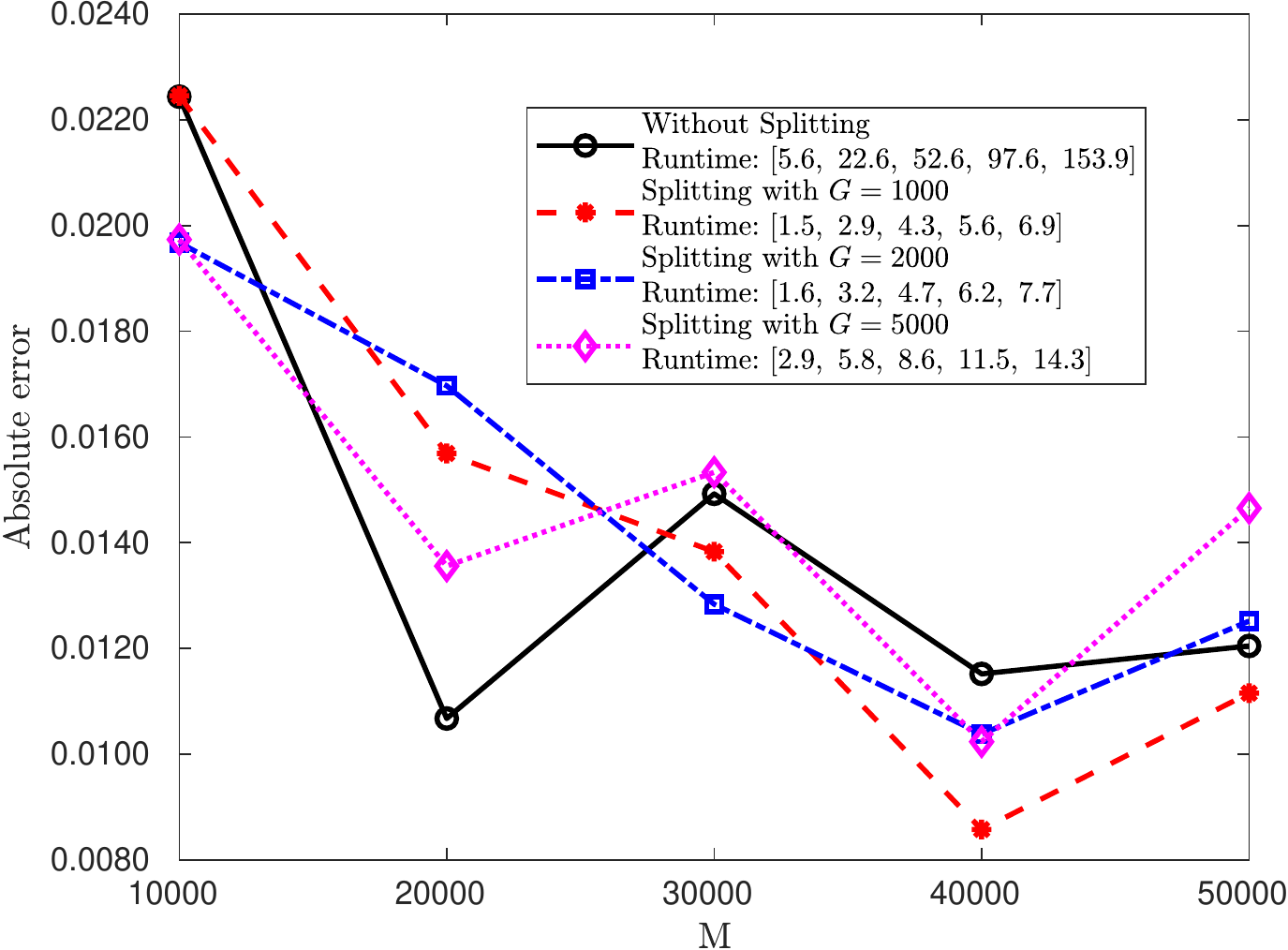}
 \subcaption{Absoulte error: $\frac{1}{10}\sum_{k=1}^{10} |Z^{\Delta_t}_{0,k}-Z_0|$}
\end{subfigure}
 \caption{Comparison of absolute errors among schemes not using and using sample-splitting ( $\theta_1=\frac{1}{2}, \theta_2=1, \theta_3=\frac{1}{2}$) with different sample sizes of group $(G),$
 the average runtimes are given in seconds. }\label{fig:01}
\end{figure}
We see that there are no considerable differences between using and not using the sample-splitting for approximating $Y_0.$
And the approximations of $Z_0$ with the sample-splitting against $M$ converge in a very stable fashion.
Furthermore, the application of sample-splitting allows a much efficient computation, e.g., when $M=50000,$ the scheme without splitting used
$153.9$ seconds while it used only $6.9$ seconds by using the splitting with $G=1000.$
In the remaining of this paper we perform all the schemes always using the splitting with $G=1000,$ unless otherwise specified.

Next we study the influence of $M$ on the error. This is a good example to test performances of the tree-based approach based on different schemes by choosing $\theta_i$'s values,
since the generator $f$ is linear and the exact solutions of $(Y_T, Z_T)$ are known. For this we fix the number of steps to $2$ and test all possible values of $\theta_i.$
We find that the explicit schemes for $\theta_3=0, \theta_2=1, \theta_1=1/2, 1$ and the implicit schemes for $\theta_3=1/2, 1, \theta_2=1, \theta_1=1/2, 1$ can converge for a small $M,$ all others need a very large number $M.$ 
As an example we report the absolute errors $\frac{1}{10}\sum_{k=1}^{10}|Y^{\Delta_t}_{0,k} - Y_0|,$ $\frac{1}{10}\sum_{k=1}^{10}|Z^{\Delta_t}_{0,k} - Z_0|$ and
the empirical standard deviations $\sqrt{\frac{1}{9}\sum_{k=1}^{10}|Y^{\Delta_t}_{0,k} - Y^{\Delta_t}_0|^2},$
$\sqrt{\frac{1}{9}\sum_{k=1}^{10}|Z^{\Delta_t}_{0,k} - Z^{\Delta_t}_0|^2}$ for some chosen schemes in Table \ref{table:01}. 
\begin{table}
	\centering 
	\begin{tabular}{c*{8}{c}r}
		\hline
		$N_T$& \multicolumn{7}{c}{$2$} \\
		\hline
		$M$ & $2000$ & $5000$ & $10000$ & $50000$ & $100000$ & $200000$ & $300000$   \\
		\hline
		& \multicolumn{7}{c}{$\frac{1}{10}\sum_{k=1}^{10} |Y^{\Delta_t}_{0,k}-Y_0|$} \\
		\hline
		$(\theta_1=1, \theta_2=1, \theta_3=\frac{1}{2})$& $0.0254$ & $0.0220$ & $0.0251$ & $0.0248$ & $0.0247$ & $0.0244$ &$0.0246$   \\
		standard deviation & $0.0193$ & $0.0147$ & $0.0099$ & $0.0023$ & $0.0021$ & $0.0016$ &8.3067e-04   \\
		\hline
		$(\theta_1=\frac{1}{2}, \theta_2=1, \theta_3=\frac{1}{2})$ & $0.0177$ & $0.0128$ & $0.0125$ & $0.0123$ & $0.0121$ & $0.0118$ &$0.0120$  \\
		standard deviation & $0.0196$ & $0.0151$ & $0.0102$ & $0.0023$ & $0.0021$ & $0.0017$ &9.1699e-04   \\
		\hline
		$(\theta_1=\frac{1}{2}, \theta_2=\frac{1}{2}, \theta_3=\frac{1}{2})$ & $0.0169$ & $0.0129$ & $0.0073$ & $0.0020$  & $0.0019$  & $0.0017$  & 7.2826e-04   \\
		standard deviation & $0.0197$ & $0.0162$ & $0.0113$ & $0.0025$ & $0.0022$ & $0.0019$ &$0.0011$   \\
		\hline
			$(\theta_1=1, \theta_2=1, \theta_3=1)$ & $0.0171$ & $0.0124$ & $0.0070$ & $0.0022$  & $0.0020$& $0.0019$  & $0.0017$   \\
		standard deviation & $0.0197$ & $0.0159$ & $0.0110$ & $0.0024$ & $0.0021$ & $0.0019$ &$0.0010$   \\
		\hline
		\hline
		& \multicolumn{7}{c}{$\frac{1}{10}\sum_{k=1}^{10} |Z^{\Delta_t}_{0,k}-Z_0|.$} \\
		\hline
		$(\theta_1=1, \theta_2=1, \theta_3=\frac{1}{2})$ & $0.1221$ & $0.1210$ & $0.1190$ & $0.1166$ & $0.1187$ & $0.1197$ &$0.1201$   \\
		standard deviation & $0.0303$ & $0.0199$ & $0.0147$ & $0.0079$ & $0.0037$ & $0.0032$ &$0.0025$  \\
		\hline
		$(\theta_1=\frac{1}{2}, \theta_2=1, \theta_3=\frac{1}{2})$ & $0.0579$ & $0.0578$ & $0.0562$ & $0.0537$ & $0.0561$ & $0.0575$ &$0.0578$  \\
		standard deviation &  $0.0319$ & $0.0235$ & $0.0165$ & $0.0081$ & $0.0042$ & $0.0036$ &$0.0027$ \\
		\hline
		$(\theta_1=\frac{1}{2}, \theta_2=\frac{1}{2}, \theta_3=\frac{1}{2})$&$0.0991$ & $0.0453$ & $0.0295$ & $0.0158$ & $0.0144$ & $0.0074$ &$0.0058$   \\
		standard deviation & $0.1111$ & $0.0550$ & $0.0312$ & $0.0171$ & $0.0173$ & $0.0077$ &$0.0060$   \\
		\hline
			$(\theta_1=1, \theta_2=1, \theta_3=1)$&$0.1114$ & $0.1111$ & $0.1095$ & $0.1072$ & $0.1095$ & $0.1107$ &$0.1112$   \\
		standard deviation & $0.0300$ & $0.0230$ & $0.0151$ & $0.0079$ & $0.0042$ & $0.0035$ &$0.0028$   \\
		\hline
	\end{tabular}
	\caption{Comparison of absolute errors for $N_T=2$ against the sample size $M.$}\label{table:01}
\end{table}
We observe, even for $N_T=2,$ the second-order scheme $(\theta_1=\frac{1}{2}, \theta_2=\frac{1}{2}, \theta_3=\frac{1}{2})$
converges only for a quite large $M.$ In particular, the error $|Z_0-Z^{\Delta_t}_0|$ approaches the convergence value first from $M=200000.$
Since error for the scheme $(\theta_1=\frac{1}{2}, \theta_2=1, \theta_3=\frac{1}{2})$ is smallest of all the schemes, which converge for a small value of $M.$ This is the reason why we will consider the scheme for $(\theta_1=\frac{1}{2}, \theta_2=1, \theta_3=\frac{1}{2})$ for the following analysis and almost all the examples.  
To take this a step further, we fix now $M=200000$ and plot the absolute error against the number of steps in Figure \ref{fig:011} when using $(\theta_1=\frac{1}{2}, \theta_2=1, \theta_3=\frac{1}{2}).$
\begin{figure}[htbp!]
 \centering
 \begin{subfigure}[b]{0.43\textwidth}
 \includegraphics[width=\textwidth]{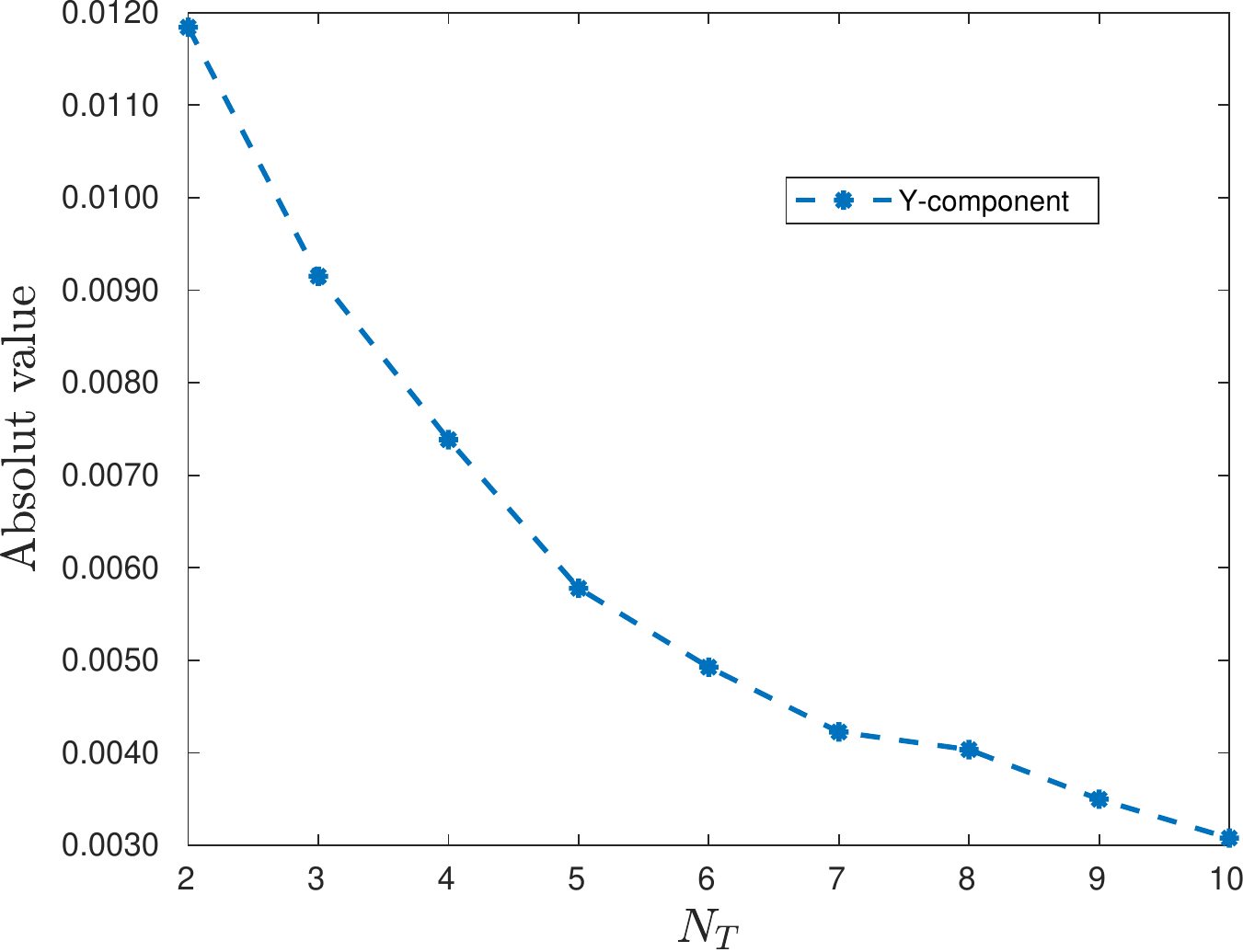}
 \subcaption{Absolute error: $\frac{1}{10}\sum_{k=1}^{10} |Y^{\Delta_t}_{0,k}-Y_0|$}
\end{subfigure}
 ~ 
 \begin{subfigure}[b]{0.43\textwidth}
 \includegraphics[width=\textwidth]{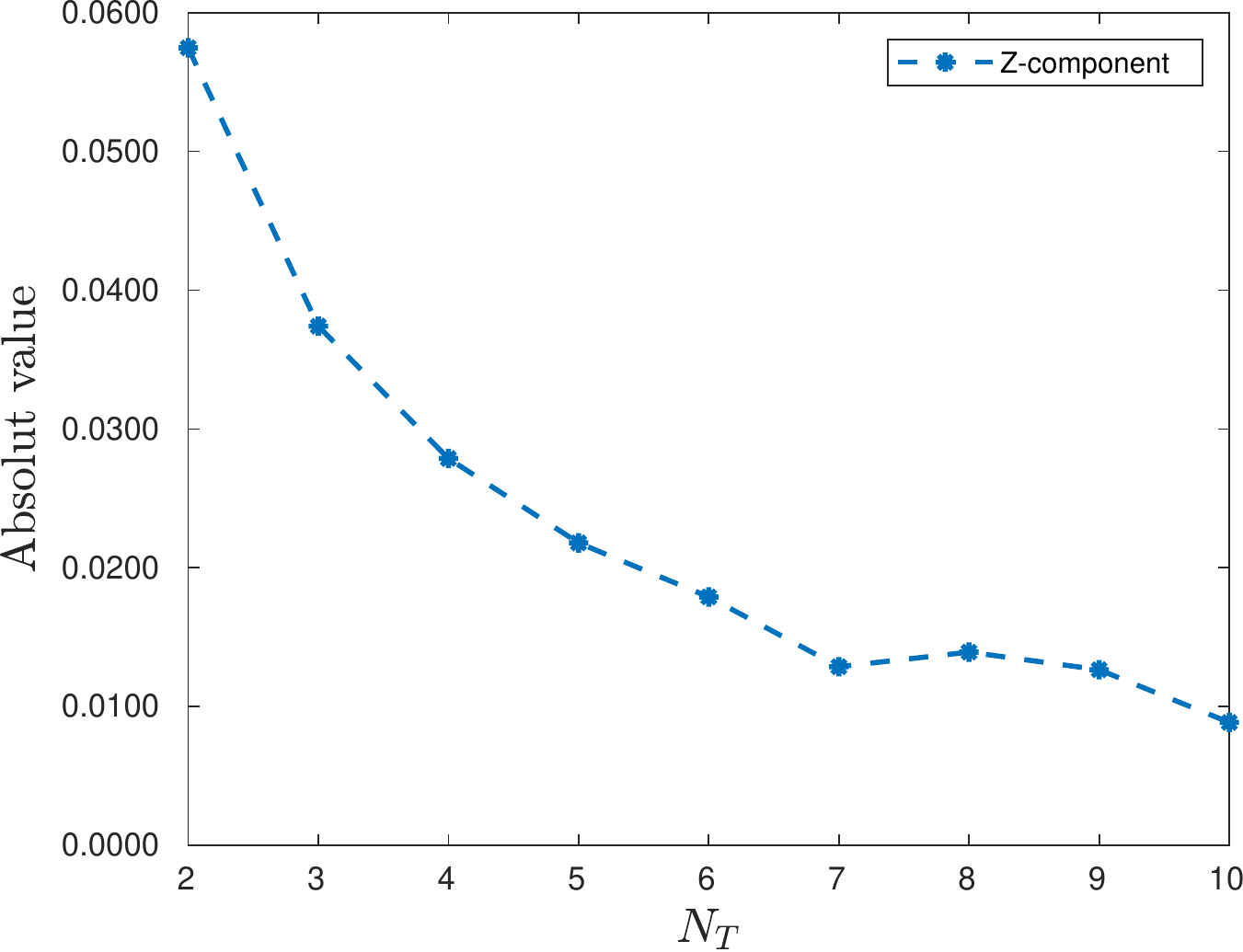}
 \subcaption{Absoulte error: $\frac{1}{10}\sum_{k=1}^{10} |Z^{\Delta_t}_{0,k}-Z_0|$}
\end{subfigure}
 \caption{Comparison of absolute errors against the number of steps $N_T$ for $\theta_1=\frac{1}{2}, \theta_2=1, \theta_3=\frac{1}{2},$ and $M=200000.$ }\label{fig:011}
\end{figure}
We see that the scheme converges meaningfully.

For the convergence with respect to the time step we refer to Figure \ref{fig:02}, where
we plot $\log_2\left(\frac{1}{10}\sum_{k=1}^{10}|Y^{\Delta_t}_{0,k} - Y_0|\right)$ and $\log_2\left(\frac{1}{10}\sum_{k=1}^{10}|Z^{\Delta_t}_{0,k} - Z_0|\right)$ with respect to $\log_2(N_T).$  To estimate
the convergence rate with respect to the time step sizes we adjust roughly sample sizes $M$ according to
the time partitions, i.e., larger $M$ for smaller $dt,$ the used sample sizes $M$ are listed in Table \ref{table:02}.
\begin{table}
\centering 
\begin{tabular}{c*{7}{c}r}
\hline
$N_T$ & 2 & 4 & 8 & 16 & 32 &   \\
\hline
$M$ & 1000 & 2000 & 20000 & 100000 & 300000 &   \\
\hline
& \multicolumn{5}{c}{$\frac{1}{10}\sum_{k=1}^{10} |Y^{\Delta_t}_{0,k}-Y_0|$}& CR \\
\hline
$(\theta_1=1, \theta_2=1, \theta_3=\frac{1}{2})$ & 0.0329 &0.0194  &  0.0080  & 0.0045  &  0.0012 & 1.17  \\
standard deviation & 0.0276  &0.0224   &  0.0051 & 0.0020 &   9.8927e-04  \\
\hline
$(\theta_1=\frac{1}{2}, \theta_2=1, \theta_3=\frac{1}{2})$ & 0.0262 &0.0174  & 0.0056 & 0.0027  &  7.9174e-04 & 1.28  \\
standard deviation & 0.0279&0.0226 &  0.0052 & 0.0020  &  9.9436e-04   \\
\hline
\hline
& \multicolumn{5}{c}{$\frac{1}{10}\sum_{k=1}^{10} |Z^{\Delta_t}_{0,k}-Z_0|$}& CR \\
\hline
$(\theta_1=1, \theta_2=1, \theta_3=\frac{1}{2})$ & 0.1142 & 0.0752 & 0.0235 & 0.0157 & 0.0092  &0.95 \\
standard deviation & 0.0273 & 0.0271  & 0.0193  & 0.0078  &  0.0055  \\
\hline
$(\theta_1=\frac{1}{2}, \theta_2=1, \theta_3=\frac{1}{2})$ & 0.0516 &0.0448 &0.0149 &0.0091 &0.0066 & 0.82  \\
standard deviation & 0.0285 &  0.0265 &  0.0180 & 0.0086  &  0.0056 \\
\hline
\hline
& \multicolumn{6}{c}{average runtime in seconds}&\\
\hline
$(\theta_1=1, \theta_2=1, \theta_3=\frac{1}{2})$ &0.1 & 0.5 & 2.3 &23.9 & 147.0  \\
\hline
$(\theta_1=\frac{1}{2}, \theta_2=1, \theta_3=\frac{1}{2})$ & 0.1  &0.2 & 2.3 & 25.5 & 144.4   \\
\hline
\end{tabular}
\caption{Absolute errors, standard deviations, average runtimes in seconds and convergence rates (CR) for the Example of BSDE \eqref{eq:example02}.}\label{table:02}
\end{table}
\begin{figure}[htbp!]
 \centering
 \begin{subfigure}[b]{0.43\textwidth}
 \includegraphics[width=\textwidth]{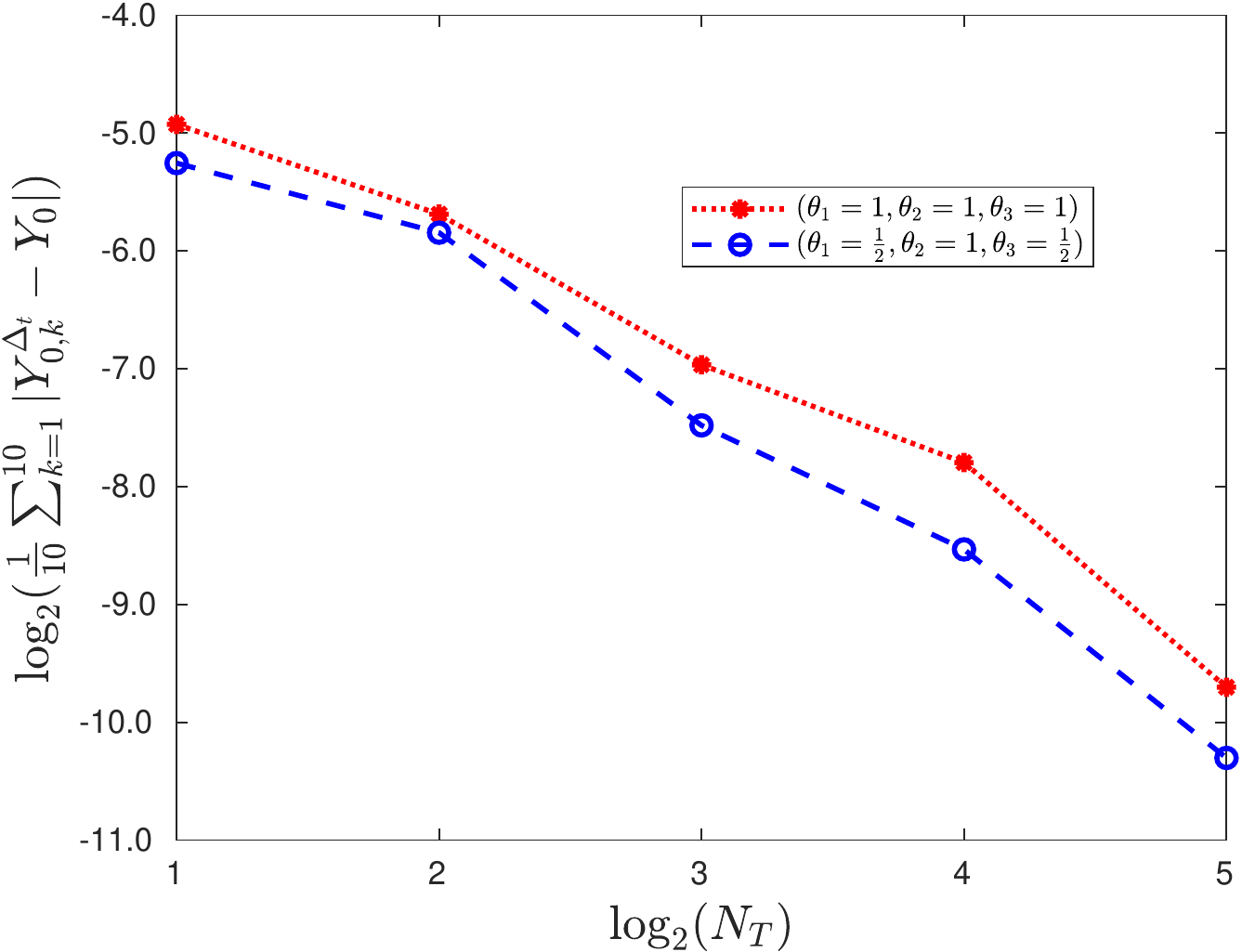}
 \subcaption{$Y$-component}
\end{subfigure}
 ~ 
 \begin{subfigure}[b]{0.43\textwidth}
 \includegraphics[width=\textwidth]{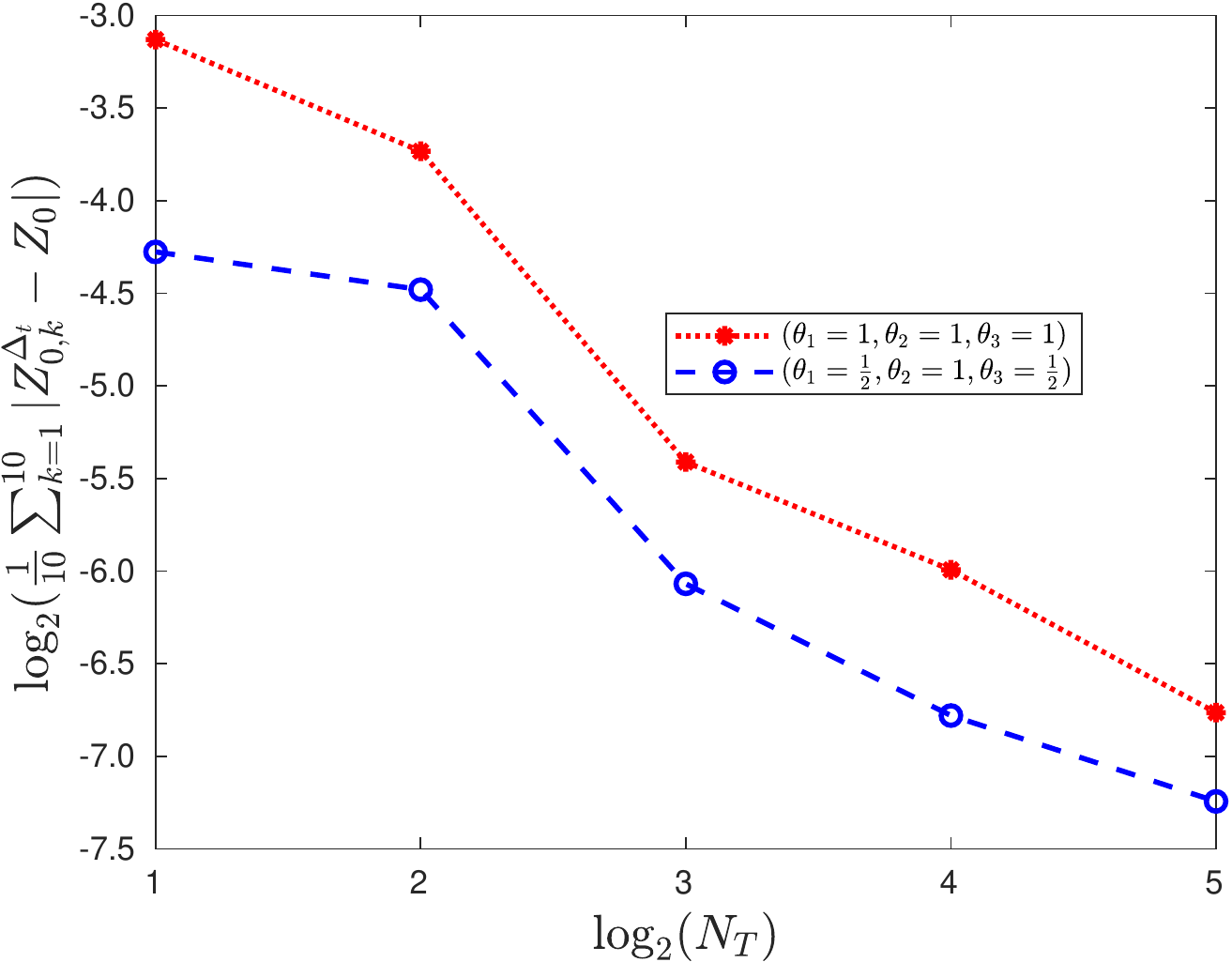}
 \subcaption{$Z$-component}
\end{subfigure}
 \caption{The plots of average of the absolute values with respect to $\log_2(N_T).$}\label{fig:02}
\end{figure}
The results shown in Table \ref{table:02} and Figure \ref{fig:01} are consistent with the conclusions in Theorem \ref{theo:convergence}.
Actually, when we use the absolute value $|Y_0-Y^{\Delta_t}_0|$
and $|Z_0-Z^{\Delta_t}_0|,$ where $Y^{\Delta_t}_0=\frac{1}{10}\sum_{k=1}^{10}Y^{\Delta_t}_{0,k}$ and $Z^{\Delta_t}_0=\frac{1}{10}\sum_{k=1}^{10}Z^{\Delta_t}_{0,k}.$
the obtained numerical convergence rates in Table \ref{table:02} are higher.
\subsection{Example of FBSDE}\label{sec:bs}
In the remaining examples we always use the scheme for $(\theta_1=\frac{1}{2}, \theta_2=1, \theta_3=\frac{1}{2}),$ unless otherwise specified.
For the example of FBSDE we compute the price of a European call option $V(t, S_t)$ via a FBSDE where the underlying
asset as the forward process, which follows a geometric Brownian motion given by
\begin{equation}
 dS_t = \mu S_t\,dt + \sigma S_t dW_t.
\end{equation}
It is well-known that the exact solution is analytically given in \cite{black1973}, namely Black-Scholes price.
We assume that the asset pays dividends with the rate $d.$ 
As introduced in \cite{Karoui1997b}, the corresponding FBSDE for the price of option can be derived by setting
up a self-financing portfolio $Y_t,$ which consists of $\pi_t$ assets and $Y_t-\pi_t$ bonds with risk-free return rate $r,$
which reads
\begin{equation}\label{eq:bsfbsde}
 \left\{
 \begin{array}{l}
\quad dS_t = \mu S_t\,dt + \sigma S_t\,dW_t,\\
-dY_t = \left(-r Y_t - \frac{\mu-r+d}{\sigma} Z_t\right)\,dt - Z_t\,dW_t,\\
\quad Y_T=\xi=\max(S_T-K, 0).
 \end{array}\right.
 \end{equation}
$Y_t$ corresponds to the option value $V(t,S_t)$, $Z_t$ is related to the hedging strategy, $Z_t=\sigma S_t \pi_t=\sigma S_t\frac{\partial V}{\partial S}.$\\
\forceindent
For $S^{\Delta_t},$ we simulate the forward process $dS_t$ by using Euler-Method, although its analytical solution is available.
Note that, although the function $g(x)=\max(x,0)$ is not differentiable in this example, we still use it to generate samples for $(Y_T, Z_T)$ in our tree-based approaches:
\begin{equation}\label{eq:algo1}
 \left\{
  \begin{array}{lcr}
 Y^{\Delta_t}_{N_T,\mathcal{M}}=\max(S^{\Delta_t}_{N_T,\mathcal{M}}-K,0),\\
 Z^{\Delta_t}_{N_T,\mathcal{M}}=\left\{\begin{array}{lcr}
                        \sigma S^{\Delta_t}_{N_T,\mathcal{M}} \quad \mbox{when}~ S^{\Delta_t}_{N_T,\mathcal{M}} >K\\
                         0, \quad \mbox{otherwise}
                         \end{array}\right.\\
 \end{array}\right.
   \end{equation} 
where $\mathcal{M}=1,\cdots,M.$
For the comparison purpose, we take the parameter values, which are used in \cite{Zhao2006}
\begin{equation}
 K=S_0=100,\,r=0.03,\,\mu=0.05,\,d=0.04,\,\sigma=0.2,\,T=0.33
\end{equation}
with the exact solution $(Y_0,Z_0)=(4.3671, 10.0950).$ For the following financial applications we 
consider the relative error $\frac{1}{10}\sum_{k=1}^{10}\frac{|Y^{\Delta_t}_{0,k}-Y_0|}{|Y_0|},$
see Table \ref{table:03}.
\begin{table}
\centering 
\begin{tabular}{c*{8}{c}r}
\hline
$N_T$        & 2 & 4 & 8 &12 & 16 & 20 &   \\
\hline
$M$          & 2000 & 10000 &30000 & 60000& 100000 & 250000 &   \\
\hline
& \multicolumn{6}{c}{$\frac{1}{10}\sum_{k=1}^{10}\frac{|Y^{\Delta_t}_{0,k}-Y_0|}{|Y_0|}$}& CR \\
\hline
$(\theta_1=\frac{1}{2}, \theta_2=1, \theta_3=\frac{1}{2})$ & 0.0230 &0.0091 &0.0052  & 0.0038 & 0.0027  &  0.0011 & 1.15  \\
standard deviation & 0.1311&0.0501 &  0.0279 & 0.0255  &  0.0143 &0.0055   \\
\hline
\hline
& \multicolumn{6}{c}{$\frac{1}{10}\sum_{k=1}^{10}\frac{|Z^{\Delta_t}_{0,k}-Z_0|}{|Z_0|}$}& CR \\
\hline
$(\theta_1=\frac{1}{2}, \theta_2=1, \theta_3=\frac{1}{2})$ &0.0492& 0.0246 &0.0151 &0.0113 &0.0091 &0.0056 & 0.86  \\
standard deviation & 0.6187 &  0.3161 &  0.1950 & 0.1307  &  0.1218 & 0.0708 \\
\hline
\hline
& \multicolumn{7}{c}{average runtime in seconds}&\\
\hline
 & 0.1  &0.5 & 3.1 & 9.5&21.3 & 66.7   \\
\hline
\end{tabular}
\caption{Relative errors, standard deviations, average runtimes in seconds and convergence rates for the Black-Scholes model.}\label{table:03}
\end{table}
Note that we have in this example the simulation error by using the Euler-Method for the forward process $dS_t.$ Furthermore, terminal condition for $Y_T$ is not differentiable at the
point $S_T=K,$ which leads to a jump for the component $Z_T$ at that points.
Although that, without any smoothing techniques we still obtain the satisfactory results using the tree-based approach.

\subsection{Example of two-dimensional FBSDE}
For the two-dimensional FBSDE we consider the Heston stochastic volatility model \cite{Heston1993} which reads
\begin{equation}\label{eq:originalHeston}
 \left\{
  \begin{array}{lcr}
 dS_t = \mu S_t\,dt + \sqrt{\nu_t} S_t\,dW^S_t,\\
 d\nu_t = \nkappa(\nmu-\nu_t)\,dt +\nsigma\sqrt{\nu_t}\,dW^{\nu}_t,\\
 dW^S_tdW^{\nu}_t=\rho\,dt,
  \end{array}
 \right.
 \end{equation}
 where $S_t$ is the spot price of the underlying asset, $\nu_t$ is the volatility. It is well-known that the Heston model \eqref{eq:originalHeston}
 can be reformulated as 
 \begin{equation}\label{eq:Heston01}
 d{\bf X}_t=  \left(\begin{array}{c}
                     d\nu_t\\
                     dS_t\\
                    \end{array}\right)= \left(\begin{array}{c}
                                          \nkappa(\nmu-\nu_t)\\
                                          \mu S_t\\
                                        \end{array}\right)\, dt + 
                                        \left(\begin{array}{c c}
                                          \nsigma\sqrt{\nu_t}&0\\
                                         S_t \rho \sqrt{\nu_t}&S_t \sqrt{1-\rho^2} \sqrt{\nu_t}\\
                                        \end{array}\right)\left(\begin{array}{c}
                     d\tilde{W}_t^{\nu}\\
                     d\tilde{W}_t^{S}\\
                    \end{array}\right),\end{equation}
 where $\tilde{W}_t^{\nu}~\mbox{and}~\tilde{W}_t^{S}$ are independent Brownian motions.
 To find the FBSDE form for the Heston Model we consider the following self-financing strategy
 \begin{align}
  dY_t &=  a_t\,dU(t,\nu_t,S_t) + b_t dS_t + c_t dP_t,\label{eq:portfolio1}\\
       &= a_t\,dU(t,\nu_t,S_t) + b_t dS_t + \frac{(Y_t - a_tU(t,\nu_t,S_t)  - b_t S_t)}{P_t} dP_t,\label{eq:portfolio2}
 \end{align}
where $U(t,\nu_t,S_t)$ is the value of another option for hedging volatility,
$dP_t=r P_t\,dt$ is used for the risk-free asset,
$a_t, b_t~\mbox{and}~c_t$ are numbers of the option, underlying asset and risk-free asset, respectively. 
We assume that 
\begin{equation}
 dU(t,\nu_t,S_t)=\eta(t,\nu_t,S_t)dt,
\end{equation}
which can be substituted into \eqref{eq:portfolio2} to obtain
 \begin{equation}\label{eq:y1}
  -dY_t = \left(a_t r U(t,\nu_t,S_t)-a_t \eta(t,\nu_t,S_t)-\frac{(\mu-r)}{\sqrt{1-\rho^2}\sqrt{\nu_t}}Z_{t,2} -r Y_t\right)\,dt-{\bf Z}_t\left(\begin{array}{c}
                     d\tilde{W}_t^{\nu}\\
                     d\tilde{W}_t^{S}\\
                    \end{array}\right)
 \end{equation}
with
\begin{equation}\label{eq:z1}
 {\bf Z}_t=(Z^1_{t},Z^2_{t})=\left(a_t\nsigma\sqrt{\nu_t}+b_t S_t \rho \sqrt{\nu_t}, b_t S_t \sqrt{1-\rho^2} \sqrt{\nu_t}\right).
\end{equation}
In the Heston model \cite{Heston1993}, the market price of the volatility risk is assumed to $\lambda \nu_t.$
With the notations used in \eqref{eq:Heston01}, the Heston pricing PDE including $\lambda$ reads
\begin{equation}\label{eq:heston2}
\frac{\partial V}{\partial t}+ rS\frac{\partial V}{\partial S } + \left(\nkappa(\nmu-\nu)-\lambda\nu\right)\frac{\partial V}{\partial \nu}
+\frac{1}{2}\nu S^2\frac{\partial^2 V}{\partial S^2}+\rho\nsigma\nu S\frac{\partial^2 V}{\partial S \partial \nu} +\frac{1}{2}\nsigma^2\nu\frac{\partial^2 V}{\partial\nu^2} -r V=0.
\end{equation}
The solution of the FBSDE \eqref{eq:y1} is exactly the solution of the Heston PDE \eqref{eq:heston2} by choosing
$r U(t,\nu_t,S_t)-\eta(t,\nu_t,S_t)\equiv-\lambda\nu_t.$
The equations \eqref{eq:y1} and \eqref{eq:z1} can thus be reformulated as
 \begin{align}\label{eq:y2}
  -dY_t &= \left(-a_t\lambda\nu_t-\frac{(\mu-r)}{\sqrt{1-\rho^2}\sqrt{\nu_t}}Z_{t}^2 -r Y_t\right)\,dt-{\bf Z}_t\left(\begin{array}{c}
                     d\tilde{W}_t^{\nu}\\
                     d\tilde{W}_t^{S}\\
                    \end{array}\right)\\
                    &=\left(-\frac{\lambda\sqrt{\nu_t}}{\nsigma}Z_{t}^1+\left(\frac{\rho\lambda\sqrt{\nu_t}}{\sqrt{1-\rho^2}\nsigma}-\frac{(\mu-r)}{\sqrt{1-\rho^2}\sqrt{\nu_t}}\right)Z_{t}^2 -r Y_t\right)\,dt-{\bf Z}_t\left(\begin{array}{c}
                     d\tilde{W}_t^{\nu}\\
                     d\tilde{W}_t^{S}\\
                    \end{array}\right)
 \end{align}
with ${\bf Z}_t$ defined in \eqref{eq:z1}.
Note that the generator in this example can be not Lipschitz continuous.
The European-style option can be replicated by hedging this portfolio. We consider e.g., a call option whose value at time $t$ is same to the portfolio value $Y_t,$
and $Y_T=\xi=\max(S_T-K, 0).$ Hence, $Y_t$ is the Heston option value $V(t,\nu_t,S_t)$, ${\bf Z}_t$ presents the hedging strategies, where 
$Z_{t}^1=\frac{\partial V}{\partial \nu}\nsigma\sqrt{\nu_t}+\frac{\partial V}{\partial S} S_t \rho \sqrt{\nu_t}$ and $Z_{t}^2=\frac{\partial V}{\partial S} S_t \sqrt{1-\rho^2} \sqrt{\nu_t}.$
The semi-analytical solution of the Heston model is available, the corresponding Delta hedging $\frac{\partial V}{\partial S}$ can thus be obtained also in a closed form.
However, the Vega hedging against volatility risk is defined as the derivative of option value with respect to the volatility $\nu_t,$ which is driven by the Cox-Ingersoll-Ross process in the Heston model
and thus not analytically available.
For this reason we can only consider the approximation of $Y$-component, namely the option price in the Heston model.
The parameter values used for this numerical test are
\begin{equation*}
\begin{split}
 &K=S_0=50,\,r=0.03,\,\mu=0.05,\,\lambda=0,\,T=0.5,\\
 &\nu_0=\nmu=0.04,\,\nkappa=1.9,\,\nsigma=0.1,\,\rho=-0.7,
\end{split}
\end{equation*}
which give the exact solution $Y_0=3.1825.$
The forward processes $dS_t~\mbox{and}~d\nu_t$ are simulated using the Euler-method, for the final values at the maturity $T$ we take
\begin{equation}\label{eq:algo2}
 \left\{
  \begin{array}{lcr}
 Y^{\Delta_t}_{N_T,\mathcal{M}}=\max(S^{\Delta_t}_{N_T,\mathcal{M}}-K,0),\\
 Z^{1,\Delta_t}_{N_T,\mathcal{M}}=\left\{\begin{array}{lcr}
                        S^{\Delta_t}_{N_T,\mathcal{M}}\rho\sqrt{\nu^{\Delta_t}_{N_T,\mathcal{M}}} \quad \mbox{when}~ S^{\Delta_t}_{N_T,\mathcal{M}} >K\\
                         0, \quad \mbox{otherwise}
                         \end{array}\right.\\
 Z^{2,\Delta_t}_{N_T,\mathcal{M}}=\left\{\begin{array}{lcr}
                         S^{\Delta_t}_{N_T,\mathcal{M}}\sqrt{1-\rho^2}\sqrt{\nu^{\Delta_t}_{N_T,\mathcal{M}}} \quad \mbox{when}~ S^{\Delta_t}_{N_T,\mathcal{M}} >K\\
                         0, \quad \mbox{otherwise}
                        \end{array} \right.\\
\end{array}\right.
   \end{equation} 
where $\mathcal{M}=1,\cdots,M.$ 
The corresponding relative errors are reported in Table \ref{table:04}.
\begin{table}
\centering 
\begin{tabular}{c*{7}{c}r}

\hline
\multicolumn{7}{c}{Exact price: 3.1825}\\
\hline
\hline
I & \multicolumn{5}{c}{The tree-based approach $(\theta_1=\frac{1}{2}, \theta_2=1, \theta_3=\frac{1}{2})$ }&  \\
\hline
$N_T$ & 2 & 4 & 8 & 16 & 32 &   \\
\hline
$M$ & 5000 & 10000 & 40000 & 100000 & 300000 &   \\
\hline
$\frac{1}{10}\sum_{k=1}^{10}\frac{|Y^{\Delta_t}_{0,k}-Y_0|}{|Y_0|}$& 0.0207 &0.0115  & 0.0043 & 0.0028  &  0.0015 & CR $\approx $0.96  \\
standard deviation & 0.0840&0.0307 &  0.0173 & 0.0109  &  0.0051   \\
 average runtime & 0.2  &0.9 & 7.4 & 38.8 & 241.3   \\
\hline
\hline
II & \multicolumn{5}{c}{The CS Crank-Nicolson ADI scheme}&  \\
\hline
$N_T$ & 2 & 4 & 8 & 16 & 32 &   \\
\hline
$\frac{|Y^{\Delta_t}-Y_0|}{|Y_0|}$& 0.0900 &0.0103  & 0.0068 & 0.0062  &  \underline{{\bf 0.0061}} &  \\
 runtime & 0.2  & 0.7 & 1.2 & 2.7 & \underline{{\bf 6.2}}   \\
\hline
\hline
III & \multicolumn{5}{c}{The tree-based approach $(\theta_1=\frac{1}{2}, \theta_2=1, \theta_3=\frac{1}{2})$ }&  \\
\hline
$N_T$ & 8 & 8 & 8 & 8 & 8 &   \\
\hline
$M$ & 100 & 500 & 1000 & 5000 & 10000 &   \\
\hline
$\frac{1}{10}\sum_{k=1}^{10}\frac{|Y^{\Delta_t}_{0,k}-Y_0|}{|Y_0|}$& 0.1181 &0.0612  & 0.0278 & 0.0162 & \underline{{\bf 0.0051}}  &  \\
standard deviation & 0.4637&0.2137 &  0.1171 & 0.0561 & \underline{{\bf 0.0183}}&  \\
 average runtime & 0.1  &0.2 & 0.2 & 1.0 & \underline{{\bf 1.9}} &    \\
\hline
\end{tabular}
\caption{Relative errors, standard deviations, average runtimes in seconds and convergence rates for the Heston model. Part I: the tree-based approach is used for different values of $N_T$ and $M;$ Part II: the CS Crank-Nicoln ADI finite different scheme is used; Part III: the tree-based approach is used for $N_T=8$ and different values of $M.$}\label{table:04}
\end{table}
We obtain quite accurate approximation for the Heston option price by solving the two-dimensional FBSDE, although the generator is not Lipschitz continuous. It is well-known that a splitting scheme of the Alternating Direction Implicit (ADI) type has been widely analyzed and applied to efficiently find the numerical solution of a two-dimensional parabolic partial differential equation (PDE). We thus compare our tree-based approach to the Craig-Sneyd (CS) Crank-Nicolson ADI finite difference scheme \cite{Craig1988} for solving the Heston model in Table \ref{table:04}.
We denote $N_S$ and $N_{\nu}$ as number of points for the stock price and the volatility grid, respectively. The ADI scheme is performed in domain $[0, 2K]$ for $S$ and $[0, 0.5]$ for $\nu$ with a uniform grid  $N_S=N_{\nu}=40,$ the time steps $N_T$ are given in Table \ref{table:04}. One can observe that the tree-based approach gives a better at least compatible result.
\subsection{Example of high-dimensional FBSDE}
It is interesting for us to test performance of tree-based approach in solving high-dimensional FBSDE. 
For this we consider the pricing problem of Rainbow option \cite{Stulz1982, Johnson1987}.
We suppose that $D$ stocks, which are for simplicity assumed to be independent and identically distributed, and driven
by 
\begin{equation}
 dS_{t,d} = \mu S_{t,d}\,dt + \sigma S_{t,d} dW_{t,d},\quad d=1,\cdots D,
\end{equation}
where $\sigma>0$ and $\mu \in \mathbb{R}.$ For the terminal condition we take that of a Call on max
\begin{equation}\label{eq:terCon}
 Y_T=\xi=\max\left(\max_{d=1,\cdots, D}(S_{T,d})-K, 0\right).
\end{equation}
The driver $f$ is then defined by
\begin{equation}\label{eq:linGenerator}
 f(t, x, y, z)=-r y - \frac{\mu-r}{\sigma} \sum_{d=1}^D z_d.
\end{equation}
In this linear example we take
\begin{equation*}
 K=S_0=100,\,r=0.04,\,\mu=0.06,\,T=0.1.
\end{equation*}
To the best of our knowledge, there is no method available for pricng the high-dimensional Rainbow option, which could allow for a less computational time than direct Monte-Carlo simulation. However, our aim is to show performance of the tree-based approach for pricing a high-dimensional Rainbow option based on the BSDE. Therefore, we compare our approach to the multilevel Monte Carlo method based on Picard approximation proposed in \cite{E2019}. The reference prices are computed with $7$ Picard iterations.

\forceindent
We consider the $10$-dimensional pricing problem, i.e., $D=10.$ 
Firstly, in Table \ref{table:05} (Part I), we adjust roughly sample sizes M to approximate the convergence rate with respect to the time step sizes. 
All the relative errors, empirical standard deviation and
convergence rate are reported there.
\begin{table}[h!]
\centering 
\begin{tabular}{c*{8}{c}r}
\hline
Reference price&\multicolumn{7}{c}{ $Y_0=10.4689$ (average runtime 2249.6 seconds) } \\
\hline
\hline
I&\multicolumn{7}{c}{The tree-based approach $(\theta_1=\frac{1}{2}, \theta_2=1, \theta_3=\frac{1}{2})$} \\
\hline
$N_T$ & 2 & 4 & 8 & 12 & 16 & 20 &   \\
\hline
$M$ & 5000 & 10000 & 80000 & 100000 & 200000 & 400000&  \\
\hline
$\frac{1}{10}\sum_{k=1}^{10}\frac{|Y^{\Delta_t}_{0,k}-Y_0|}{|Y_0|}$ & 0.0390 &0.0195 & 0.0078 & 0.0038  &  0.0013 &3.4737e-04 & CR $\approx$1.9  \\
standard deviation & 0.0429&0.0356 &  0.0109 & 0.0080  &  0.0045 & 0.0025  \\
average runtime in seconds& 0.9  &4.3 & 75.7 & 146.6 & 602.9 & 999.2  \\
\hline
\hline
II&\multicolumn{7}{c}{The tree-based approach $(\theta_1=\frac{1}{2}, \theta_2=1, \theta_3=\frac{1}{2})$} \\
\hline
$N_T$ & 12 & 12 & 12 & 12 & 12 & 12 &   \\
\hline
$M$ & 100 & 500 & 1000 & 2000 & 5000 & 10000&  \\
\hline
$\frac{1}{10}\sum_{k=1}^{10}\frac{|Y^{\Delta_t}_{0,k}-Y_0|}{|Y_0|}$ & 0.0309 &0.0154 & 0.0100 &  \underline{{\bf 0.0061 }}  &  0.0059 &0.0056 &   \\
standard deviation & 0.4048&0.1721 &  0.1326 &  \underline{{\bf 0.0762 }}  &  0.0552 & 0.0303  \\
average runtime in seconds& 0.5  &0.9 & 1.6 &  \underline{{\bf 3.1 }} & 7.6 & 15.0  \\
\hline
\hline
III&\multicolumn{7}{c}{The multilevel Monte Carlo method \cite{E2019}} \\
\hline
Number of the Picard iteration & 1 & 2 & 3 & 4 & 5 & 6 &   \\
\hline
$\frac{1}{10}\sum_{k=1}^{10}\frac{|Y^{\Delta_t}_{0,k}-Y_0|}{|Y_0|}$ & 0.1920 &0.2312 & 0.0759 & 0.02290  &  0.0120 & \underline{{\bf 0.0058}}& \\
\hline
standard deviation &1.7304&2.8257 &  0.9796 & 0.2691  &  0.1695 & \underline{{\bf 0.0825}}  \\
\hline
average runtime in seconds& 0.0  & 0.0 & 0.0 & 0.3 & 3.4 & \underline{{\bf 43.9}}  \\
\hline

\end{tabular}
\caption{Relative errors, standard deviations, average runtimes in seconds and convergence rate for the max option in the case $D=10.$ Part I: the tree-based approach is used for different values of $N_T$ and $M;$ Part II: the tree-based approach is used for $N_T=12$ and different values of $M;$ Part III: the multilevel Monte Carlo method is used for different iteration numbers. }\label{table:05}
\end{table}
The reference price $Y_0=10.4689$ is computed by means of
the multilevel-Picard approximation method in \cite{E2019} with 7 Picard iterations, whereas the average runtime are 2249.6 seconds.
It is not difficult to see that our results are quite promising, and show that the $10$-dimensional problem
can be highly effective and accurate approximated using the tree-based approach. The obtained convergence rate of the proposed scheme is 1.9. 
For a comparison purpose, using the same reference price we report the errors, standard deviations and average rumtimes for the Picard iteration number $\{1,\cdots,6\}$
using the method in \cite{E2019} in Table \ref{table:05} (Part III).
To compare the result for the Picard iteration number equals $6$ (bold and underlined), in Table \ref{table:05} (Part II) we show our results for $N_T=12$ by varying different sample sizes.
From our result for $N_T=12~\mbox{and}~M=2000$ (bold and underlined) we see that for this $10$-dimensional pricing problem, our scheme is more than 10 times faster than the approximation method in \cite{E2019}. 
Note that, in order to see performance of our approach for the problem in which the forward SDE does not exhibit an analytical solution, we simply use the Euler method for $dS.$

\forceindent
Finally, we test our scheme for the $100$-dimensional pricing problem. 
\begin{table}[h!]
\centering 
\begin{tabular}{c*{8}{c}r}
\hline
reference price&\multicolumn{7}{c}{ $Y_0=17.4267$ (average runtime 2613.9 seconds) } \\
\hline
\hline
I&\multicolumn{7}{c}{The tree-based approach $(\theta_1=\frac{1}{2}, \theta_2=1, \theta_3=\frac{1}{2})$} \\
\hline
$N_T$ & 2 & 4 & 8 & 12 & 16 & 20 &   \\
\hline
$M$ & 5000 & 10000 & 80000 & 100000 & 200000 & 300000  \\
\hline
$\frac{1}{10}\sum_{k=1}^{10}\frac{|Y^{\Delta_t}_{0,k}-Y_0|}{|Y_0|}$ & 0.1920 &0.0943  & 0.0466 & 0.0297  &0.0206 & 0.0152 &CR$\approx$1.09  \\
standard deviation & 0.0771&0.0353 &  0.0180 & 0.0111  &  0.0104 &0.0082   \\
\hline
 average runtime in seconds & 16.2  & 90.1 & 1621 & 3162 & 8529 & 16180   \\
\hline
\hline
II&\multicolumn{7}{c}{The tree-based approach $(\theta_1=\frac{1}{2}, \theta_2=1, \theta_3=\frac{1}{2})$} \\
\hline
$N_T$ & 20 & 20 & 20 & 20 & 20 & 20 &   \\
\hline
$M$ & 100 & 500 & 1000 & 2000 & 5000 & 10000&  \\
\hline
$\frac{1}{10}\sum_{k=1}^{10}\frac{|Y^{\Delta_t}_{0,k}-Y_0|}{|Y_0|}$ & 0.0240 & \underline{{\bf 0.0140}} & 0.0159  & 0.0110  & 0.0143 &0.0137 &   \\
standard deviation & 0.4819& \underline{{\bf 0.2275}} & 0.1858  & 0.0962  & 0.0761 & 0.0704  \\
average runtime in seconds& 8.6  & \underline{{\bf 27.6}}& 56.1 & 111.3 & 270.2&  541.9 \\
\hline
\hline
III&\multicolumn{7}{c}{The multilevel Monte Carlo method \cite{E2019}} \\
\hline
Number of the Picard iteration & 1 & 2 & 3 & 4 & 5 & 6 &   \\
\hline
$\frac{1}{10}\sum_{k=1}^{10}\frac{|Y^{\Delta_t}_{0,k}-Y_0|}{|Y_0|}$ & 0.1970 &0.1368 & 0.0606 & 0.0546  &  0.0249 &\underline{{\bf 0.0165}}& \\
\hline
standard deviation &4.2130&3.1551 &  1.4108 & 1.2591  &  0.4458 & \underline{{\bf 0.3539}}  \\
\hline
average runtime in seconds& 0.0  & 0.0 & 0.0 & 0.4 & 4.0 & \underline{{\bf 50.1}}  \\
\hline
\end{tabular}
\caption{Relative errors, standard deviations, average runtimes in seconds and convergence rate for the max option in the case $D=100.$
Part I: the tree-based approach is used for different values of $N_T$ and $M;$ Part II: the tree-based approach is used for $N_T=20$ and different values of $M;$ Part III: the multilevel Monte Carlo method is used for different iteration numbers.}\label{table:06}
\end{table}
Note that due to the limitation of memory, we only set $M=300000$ for $N_T=20$ in the $100$-dimensional case.
In Table \ref{table:06}, the average runtime of using the multilevel-Picard method for $100$-dimension (2613.9) is not much longer than
that (2249.6) in Table \ref{table:05} for $10$-dimension. Especially, by comparing the average runtime in Table 5 in Section 4.3 in \cite{E2019} for $1$-dimensional to that in Table 6 in the same section in \cite{E2019}, it seems that the multilevel-Picard method in \cite{E2019}
is not sufficiently efficient for a lower dimensional problem. In contrast,
in the previous numerical experiments ($10$-dimensional problem) we have seen that our proposed approach is much more efficient.
Although the computational expense in our proposed approach increases for the increasing dimensionality,
for this $100$-dimensional pricing problem our approach is still two time faster than the method proposed in \cite{E2019} for the same or better error level,
see both the results which are bold and underlined in Table \ref{table:06}.
The proposed scheme converges with the rate of 1.09 for the $100$-dimensional pricing problem.

\subsection{Example of nonlinear FBSDE}
In this section we test our scheme for nonlinear high-dimensional problems. 
We find that nonlinear training data may lead to overfitting when directly using the above introduced procedure.
Therefore, to avoid the overfitting for the nonlinear problems, we propose to control the error already while growing a tree.
For this, we estimate the cross validation mean squared errors of the trees, which are constructed with different number of observations in each branch node.
Clearly, the best tree, namely the best number of observations in each branch node can be determined by comparing the errors.
Theoretically, for the best result, the error control needs to be performed for each time step. However, it will be computationally too expensive.
Fortunately, in our test we find the best numbers of observations for each time step are very close to each other. 
For substantially less computation time, one only needs determine one of them, e.g.,, for the first iteration, and fix it for all other iterations.
We note that, the pruning procedure cannot bring considerable improvement when the tree has been grown using the best number of observations, is thus
unnecessary in this case.

As an example, we consider a pricing problem of an European option in a financial market with different interest rate
for borrowing and lending to hedge the European option. This pricing problem is analyzed in \cite{Bergman2008} and is used as a standard nonlinear (high-dimensional)
example in the many works, see e.g., \cite{E2017, E2019, Gobet2005, Bender2017}. Similar but different to \eqref{eq:terCon} and \eqref{eq:linGenerator},
the terminal condition and generator for the option pricing with different interest rate read as
\begin{equation}\label{eq:terCon_non_lin}
 Y_T=\xi=\max\left(\max_{d=1,\cdots, D}(S_{T,d})-K_1, 0\right) - 2 \max\left(\max_{d=1,\cdots, D}(S_{T,d})-K_2, 0\right)
\end{equation}
and
\begin{equation}\label{eq:nonLinGenerator}
 f(t, x, y, z)=-R^l y - \frac{\mu-R^l}{\sigma} \sum_{d=1}^D z_d + (R^b-R^l)\max\left(0, \frac{1}{\sigma}\sum_{d=1}^D z_d - y\right),
\end{equation}
respectively,
where $R^b, R^l$ are different interest rates and $K_1, K_2$ are strikes. Obviously, \eqref{eq:terCon_non_lin} and \eqref{eq:nonLinGenerator} are both nonlinear.

We first consider a 1-dimensional case, in which we use $Y_T=\xi=\max\left(S_{T}-100, 0\right)$ instead of \eqref{eq:terCon_non_lin} to agree with the setting in \cite{Gobet2005, E2017}.
The parameter values are set as: $T=0.5,$ $\mu=0.06,$ $\sigma=0.02,$ $R^l=0.04,$ $R^b=0.06.$ We use $Y_0=7.156$ computed
using the finite difference method as the reference price.
Note that the reference price is confirmed in \cite{Gobet2005} as well.
Firstly, we fix $M=200000,$ $G=50000$ and plot the relative error against the number of steps in Figure \ref{fig:nonlinear1d}.
\begin{figure}[htbp!]
	\centering
		\includegraphics[width=3in]{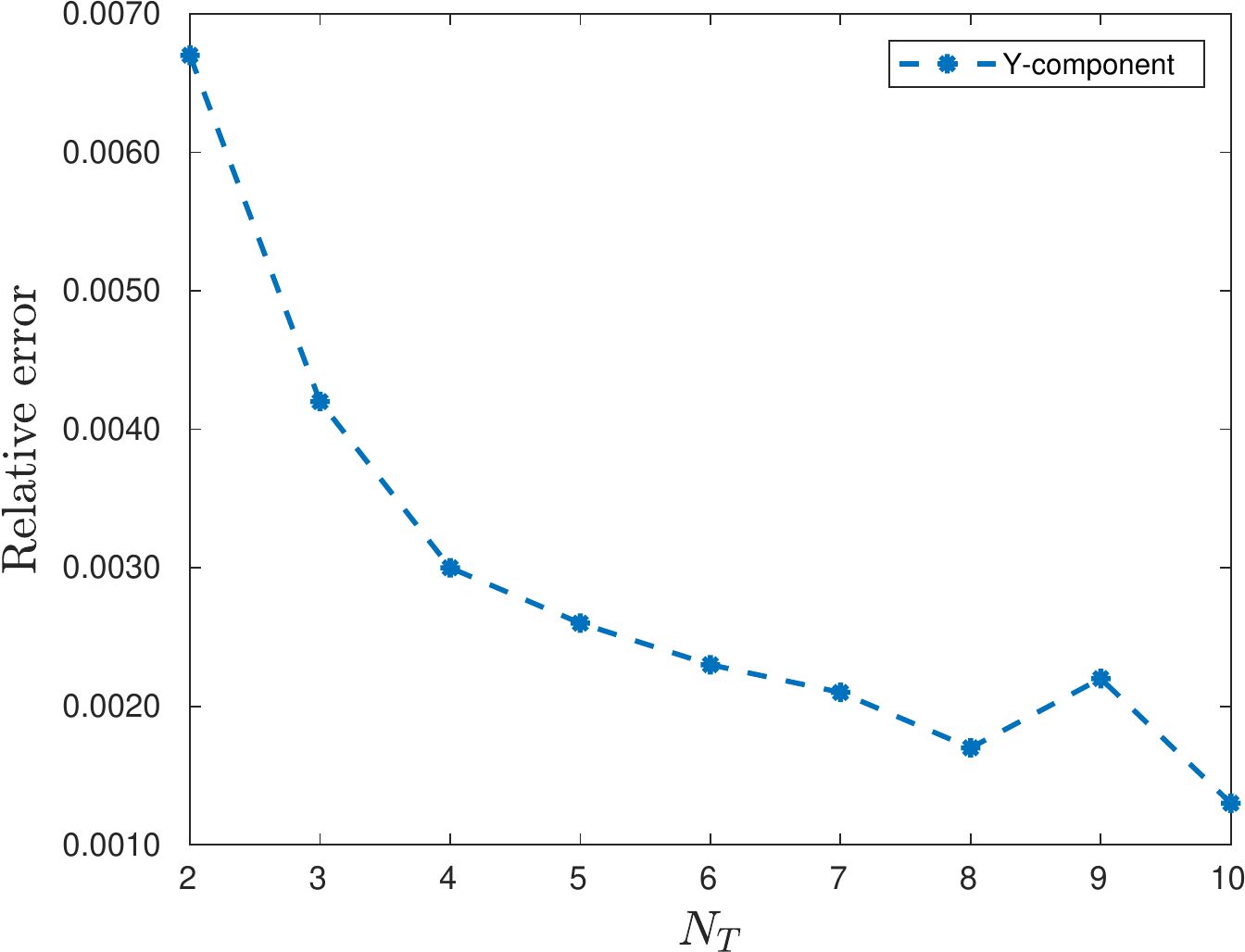}
	\caption{Comparison of relative errors against the number steps $N_T$ and $M=200000$ for one-dimensional pricing with different interest rate.}\label{fig:nonlinear1d}
\end{figure}
We obtain very good numerical results, and reach an error of order $10^{-3}.$
In Table \ref{table:07} we compare our results to the results given in Table 5 in \cite{E2019}, and show that the tree-based approach with $N_T=10$ can reach accuracy level of the multilevel Monte Carlo with 7 Picard iterations for
significantly less computational time.
\begin{table}[h!]
\centering 
\begin{tabular}{c*{8}{c}r}
\hline
Reference price&\multicolumn{7}{c}{ $Y_0=7.156$} \\
\hline
\hline
I&\multicolumn{7}{c}{The tree-based approach $(\theta_1=\frac{1}{2}, \theta_2=1, \theta_3=\frac{1}{2})$} \\
\hline
$N_T$ & 10 & 10 & 10 & 10 & 10 &10 & 10  \\
\hline
$M$ & 2000 &4000 &10000 & 20000  & 50000 & 100000& 200000 \\
\hline
$\frac{1}{10}\sum_{k=1}^{10}\frac{|Y^{\Delta_t}_{0,k}-Y_0|}{|Y_0|}$ & 0.0239&0.0152 &0.0107 & 0.0073 &  0.0043 &  0.0035& \underline{{\bf 0.0013}}\\
standard deviation & 0.2089&0.1100 &0.0953 &  0.0686   &  0.0364 & 0.0352& \underline{{\bf 0.0130}}   \\
average runtime in seconds& 0.1&0.1  &0.2 & 0.3  & 0.9 & 1.8 &\underline{{\bf3.7}} \\
\hline
\hline
II&\multicolumn{7}{c}{The multilevel Monte Carlo method, see Table 5 in \cite{E2019}} \\
\hline
Number of the Picard iteration & 1 & 2 & 3 & 4 & 5 & 6 & 7  \\
\hline
$\frac{1}{10}\sum_{k=1}^{10}\frac{|Y^{\Delta_t}_{0,k}-Y_0|}{|Y_0|}$ & 0.8285 	 &0.4417 & 0.1777 & 0.1047  &  0.0170 &  0.0086& \underline{{\bf 0.0019}}\\
\hline
standard deviation &7.7805&4.0799 & 1.6120 & 0.8106  &  0.1512 & 0.0714 & \underline{{\bf 0.0157}} \\
\hline
average runtime in seconds& 0.0  & 0.0 & 0.0 & 0.3 & 3.1 & 38.7&\underline{{\bf 1915.1}}  \\
\hline
\end{tabular}
\caption{Relative errors, standard deviations, average runtimes in seconds and convergence rate for the 1-dimensional pricing with different interest rates.
Part I: the tree-based approach is used for $N_T=10$ and different values of $M;$ Part II: the multilevel Monte Carlo method is used for different iteration numbers. }\label{table:07}
\end{table}
Note that the samples-splitting $(G=50000)$ is only used for $M=100000, 200000.$ 
Finally, we test our scheme for 100-dimensional nonlinear pricing problem.
In contrast to the case of 1-dimension, the terminal condition \eqref{eq:terCon_non_lin} is more challenge to deal with. In our test,
\eqref{eq:terCon_non_lin} can still be used to generate samples of $Y_{N_T}.$ However, for $Z_{N_T},$ the one-sided derivative of \eqref{eq:terCon_non_lin}, as it in
\eqref{eq:algo1} and \eqref{eq:algo2} is not sufficient for the 100-dimensional nonlinear pricing problem. Therefore, for this example we choose the scheme by setting $\theta_1=\theta_2=\theta_3=1$ such that
$Z$-component will be not directly needed for the iterations.
In Figure \ref{fig:nonlinear100d}, the results of using $M=200000,$ $G=50000$ against the number steps $N_T$ are reported.
\begin{figure}[htbp!]
	\centering
	\includegraphics[width=3in]{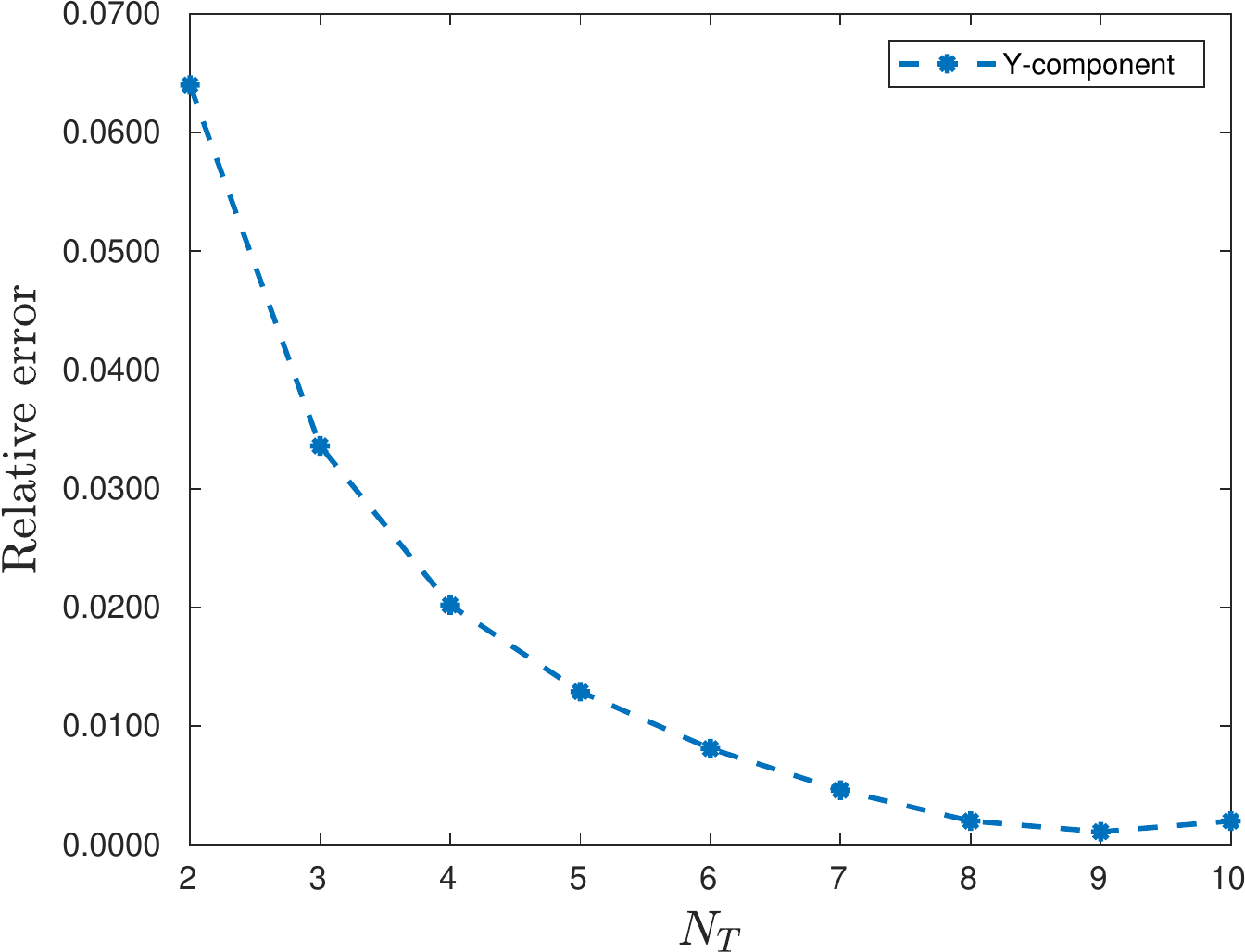}
	\caption{Comparison of relative errors against the number steps $N_T$ and $M=200000$ for 100-dimensional pricing with different interest rate. }\label{fig:nonlinear100d}
\end{figure}
Again, in Table \ref{table:08} we compare our results to them in Table 6 in \cite{E2019}. The reference price $Y_0=21.2988$ is computed using the multilevel Monte Carlo with 7 Picard iterations,
whereas $K_1=120, K_2=150,$ and values of other parameters are the same as those for the 1-dimensional case. We only use the samples-splitting $(G=50000)$ when $M > 50000.$ 
\begin{table}[h!]
	\centering 
	\begin{tabular}{c*{8}{c}r}
		\hline
		Reference price&\multicolumn{7}{c}{$Y_0=21.2988$(average runtime 2725.1 seconds)} \\
		\hline
		\hline
		I&\multicolumn{7}{c}{The tree-based approach $(\theta_1=1, \theta_2=1, \theta_3=1)$} \\
		\hline
		$N_T$ & 10 & 10 & 10 & 10 & 10 & 10 & \\
		\hline
		$M$ & 10000 & 50000 & 100000 & 200000 & 300000 & 400000& \\
		\hline
		$\frac{1}{10}\sum_{k=1}^{10}\frac{|Y^{\Delta_t}_{0,k}-Y_0|}{|Y_0|}$ &\underline{{\bf 0.0212}} &0.0022 & 0.0023 & 0.0020  &  0.0017 &  0.0022&\\
		standard deviation &\underline{{\bf 0.0629}}&0.0286 &  0.0243 & 0.0199  &  0.0143 & 0.0129&  \\
		average runtime in seconds&\underline{{\bf 19.6}}  &115.9 & 233.3 & 464.9 & 703.5 & 947.7& \\
		\hline
		\hline
		II&\multicolumn{7}{c}{The multilevel Monte Carlo method, see Table 6 in \cite{E2019}} \\
		\hline
		Number of the Picard iteration & 1 & 2 & 3 & 4 & 5 & 6 &  \\
		\hline
		$\frac{1}{10}\sum_{k=1}^{10}\frac{|Y^{\Delta_t}_{0,k}-Y_0|}{|Y_0|}$ & 0.4415 	 &0.4573 & 0.1798 & 0.1042  &  0.0509 &  \underline{{\bf 0.0474}}\\
		\hline
		standard deviation &8.7977&11.3167 & 4.4920 & 2.9533  &  1.4486 &\underline{{\bf 1.3757}} \\
		\hline
		average runtime in seconds& 0.0  & 0.0 & 0.0 & 0.4 & 4.2 & \underline{{\bf 52.9}}&  \\
		\hline
	\end{tabular}
	\caption{Relative errors, standard deviations, average runtimes in seconds and convergence rate for pricing with different interest rates in the case D=100.
	Part I: the tree-based approach is used for $N_T=10$ and different values of $M;$ Part II: the multilevel Monte Carlo method is used for different iteration numbers. }\label{table:08}
\end{table}
We see that our result with $N_T=10, M=2000$ is already better than the approximation of multilevel Monte-Carlo with 6 iterations for almost same computational time.
Furthermore, a better approximation (smaller standard deviations) can always be achieved with a larger number of $M.$
Note that the same reference price is used to compare the deep learning-based numerical methods for high-dimensional BSDEs in \cite{E2017} (Table 3),
which has achieved a relative error of $0.0039$ in a runtime of $566$ seconds.
\section{Conclusion}
In this work, we have studied solving forward-backward stochastic differential equations numerically using the regression tree-based methods.
We show how to use the regression tree to approximate the conditional expectations arising by discretizing the time-integrands using the general theta-discretization method.
We have performed several numerical experiments for different types of (F)BSDEs including its application to $100$-dimensional nonlinear pricing problem.
Our numerical results are quite promising and indicate that the tree-based approach is very attractive to solve high-dimensional nonlinear (F)BSDEs.

\bibliography{mybibfile}
\bibliographystyle{apalike}
\end{document}